\documentclass{amsart}[12pt,a4paper]

\usepackage[colorlinks=true,pagebackref,hyperindex]{hyperref}
\usepackage[english]{babel} 
\usepackage{latexsym}
\usepackage{enumerate}
\usepackage{color}  
\usepackage[T1]{fontenc} 
\usepackage{graphicx}
\usepackage{graphics} 

\input xy
\xyoption{all}
\usepackage{amsfonts}  
\usepackage{amsmath}
\usepackage{amssymb} 
\usepackage{mathrsfs} 
\usepackage{amsthm}

\makeatletter
\@namedef{subjclassname@2020}{%
  \textup{2020} Mathematics Subject Classification}
\makeatother

\newtheorem{theorem}{Theorem}[section]
\newtheorem{lemma}[theorem]{Lemma}

\newtheorem{proposition}[theorem]{Proposition}
\newtheorem{corollary}[theorem]{Corollary}
\newtheorem{definition}[theorem]{Definition}
\newtheorem{remark}[theorem]{Remark}

\theoremstyle{definition}
\newtheorem{example}[theorem]{Example}

\DeclareMathOperator{\NS}{N^1}

\DeclareMathOperator{\Vol}{\mathrm{Vol}}
\DeclareMathOperator{\vol}{\mathrm{vol}}
\DeclareMathOperator{\diff}{\mathrm{Diff}}
\DeclareMathOperator{\lct}{\mathrm{lct}}

\newcommand{\qq}{\mathbb{Q}}
\newcommand{\pp}{\mathbb{P}}
\newcommand{\rr}{\mathbb{R}}
\newcommand{\D}{\mathfrak{D}}
\newcommand{\oo}{\mathcal{O}}
\newcommand{\supp}{\mathrm{Supp}}

\newcommand{\sje}{s_{n-j, \epsilon}}
\newcommand{\mje}{m_{n-j, \epsilon}}
\newcommand{\mfe}{m_{\dim F, \epsilon}}
\newcommand{\sfe}{s_{\dim F, \epsilon}}

\DeclareRobustCommand{\SkipTocEntry}[5]{}

\begin{document}


\author{Gabriele Di Cerbo}
\email{dicerbo@math.princeton.edu}
\address{Department of Mathematics, Princeton University, Princeton, NJ 08540, USA}

\author{Roberto Svaldi}
\email{roberto.svaldi@epfl.ch}
\address{EPFL, SB MATH-GE, MA B1 497 (B\^{a}timent MA), Station 8, CH-1015 Lausanne, Switzerland.}


\title[Birational boundedness of elliptic Calabi--Yau varieties]{Birational boundedness of low dimensional elliptic Calabi--Yau varieties with a section}






\thanks{
GDC is partially supported by the Simons Foundation and the NSF Grant DMS-1702358.
Most of this work was completed during several visits of RS to Columbia University. 
He would like to thank Columbia University for the hospitality and the nice working environment. 
He would also like to thank MIT where he was a graduate student 
and UCSD where he was a visitor when part of this work was completed. 
He kindly acknowledges financial support from NSF research grant no: 1200656, no: 1265263 and Churchill College, Cambridge. 
During the final revision of this work he was supported by 
funding from the European Research Council under the European Union's Seventh Framework Programme (FP7/2007-2013)/ERC Grant agreement no. 307119.
}


\begin{abstract}
We prove that there are finitely many families, up to isomorphism in codimension one, 
of elliptic Calabi-Yau manifolds $Y\rightarrow X$ with a rational section, provided that $\dim(Y)\leq 5$ and $Y$ is not of product-type. 
As a consequence, we obtain that there are finitely many possibilities for the Hodge diamond of such manifolds. 
The result follows from log birational boundedness of klt pairs $(X, \Delta)$ with $K_X+\Delta$ numerically trivial
and not of product-type, in dimension at most $4$.
\end{abstract}


\maketitle


\section{Introduction}
Smooth varieties with trivial canonical bundle constitute one of the fundamental building blocks in the birational classification of algebraic varieties.
Hence, it is an important problem to understand their possible algebraic and topological structures. 
By a well-known result of Beauville and Bogomolov \cite{MR730926}, 
every smooth variety with trivial canonical bundle can be decomposed -- after a finite \'etale cover -- as a product of abelian, hyperk\"ahler or Calabi--Yau varieties. 
A smooth projective variety $Y$ with trivial canonical bundle is Calabi--Yau if it is simply connected and $H^i(Y, \mathcal{O}_Y)=0$ for $0<i<\dim Y$.

While we have a clear understanding of the algebraic and topological structures of abelian varieties, the situation is far from being settled in the other two cases.
In dimension two, Calabi--Yau surfaces are K3 surfaces and they admit a unique topological model, though in order to show this one has to take into account non-projective models. 
In higher dimension the problem of determining the number of
deformation types of Calabi--Yau manifolds is wide open;
nonetheless, every Calabi-Yau manifold of dimension at least three is automatically projective.

The problem of understanding whether there are finitely many algebraic families of Calabi--Yau varieties for each fixed dimension $n\geq 3$ is the main motivation of this paper.
A class of varieties that can be described by finitely many families is said to be {\it bounded}, see \S~\ref{bound.cyp.sect}.
In particular, we study the boundedness of those Calabi--Yau varieties admitting an elliptic fibration, i.e. a projective morphism with connected fibers, whose general fiber is one-dimensional of genus $1$.
This class of varieties has already been understood in dimension $\leq 3$ and it is of particular interest for its applications in theoretical physics, as we explain below.

In dimension 2, the question of boundedness has a negative answer, since there are infinitely many algebraic families of projective elliptic $K3$ surfaces. 
On the other hand, if we consider elliptic $K3$ surfaces admitting a section, it is not hard to show that boundedness holds for this class of varieties.

Surprisingly, the lack of boundedness for general Calabi--Yau elliptic fibrations appears to be a behaviour peculiar to surfaces only: M. Gross showed that there are indeed finitely many families, up to birational equivalence, of Calabi--Yau threefolds possessing a non-isotrivial elliptic fibration over a rational base, see \cite{Gross}.
In the three-dimensional case, when the base of the fibration is not rational, then it is an Enriques surface and the total space behaves like a product, the fibration being isotrivial. 
In his work, Gross showed that the problem of boundedness can be split into two steps: studying the boundedness of those surfaces that appear as bases of elliptic fibrations together with the study of elliptic curves defined over the function field of the possible bases. 
A similar behaviour is expected to hold in higher dimension.

The main result of this paper is a generalization of Gross' theorem to elliptic varieties of dimension of most $5$ with terminal singularities and numerically trivial canonical class, provided that the fibration admits a rational section.

\begin{theorem}
\label{elliptic}
Let $2\leq n\leq 5$ be an integer. Then there are finitely many families, up to isomorphism in codimension one, of elliptic fibrations $Y\rightarrow X$ such that
\begin{enumerate}
\item  $Y$ is terminal projective of dimension $n$ and $K_Y \equiv 0$,
\item $Y$ is not of product-type, and
\item there exists a rational section $X\dashrightarrow Y$.
\end{enumerate}
\end{theorem}

The condition that $Y$ be not of product-type should be thought as an analogue of the non-isotriviality required in Gross' theorem. 
For a precise definition we refer the reader to Definition~\ref{defproduct}. 
If $Y$ is of product type, then $Y$ is birationally isomorphic to the quotient of a product, cf. Theorem~\ref{ambro}.

Once we require the existence of a section, our proof shows that even elliptic K3 surfaces are bounded.
The proof of~\ref{elliptic} heavily relies on the existence of a section to produce a suitably positive divisor, more precisely a big and nef divisor, on the total space with bounded volume, see \S~\ref{boundedness.sect}-\ref{sec.proofs}. 
The need of a section is the main reason why we are not able to recover the full statement proved by Gross in dimension $3$, although we expect Gross' result to generalize to higher dimension.
After the completion of this work, Birkar and the authors obtained an extension of Theorem~\ref{elliptic} to any dimension, see~\cite{bds}, while Filipazzi and the second named author \cite{FS20} have extended the ideas contained in this paper to the case of $n$-dimensional varieties of Kodaira dimension $n-1$, cf. also \cite{Fil20}.
 
Theorem \ref{elliptic}, combined with results from motivic integration, implies uniform boundedness of Hodge numbers for elliptic Calabi--Yau manifolds. 

\begin{corollary}
\label{hodge}
Let $2 \leq n\leq 5$ be an integer. 
Then there exists a positive integer $M_n$ such that $h^{p,q}(Y)\leq M_n$ for any $p,q$, where $Y$ is a  smooth manifold satisfying the conditions of Theorem~\ref{elliptic}.
\end{corollary}

Elliptic Calabi--Yau fourfolds with a section seem to be the most relevant in F-theory and bounding their Hodge numbers is a central problem in string theory, see, for example,~\cite{TW}. 
 
The key ingredient in the proof of Theorem~\ref{elliptic} is a boundedness statement for the base of an elliptic fibration. 
If $Y$ has an elliptic fibration $Y \rightarrow X$ then by a general version of Kodaira's canonical bundle formula, $X$ carries the structure of a Calabi--Yau pair, which means that there exists an effective divisor $\Delta$ such that $(X,\Delta)$ is klt and  $K_X +\Delta$ is numerically trivial, see \S~\ref{lc.trivial.fibrations}.
In order to understand boundedness of elliptic Calabi--Yau varieties, we first need to understand if those Calabi--Yau pairs appearing as the base of such fibrations indeed belong to a bounded family. 
One crucial observation is that we have some control on the coefficients of $\Delta$, as it was already clear to Kodaira: indeed the coefficients vary in a finite set, which is the first fundamental hint and step towards boundedness, cf. Section \ref{bound.cyp.sect}.

Let us notice that the set of threefolds of the form $S\times \pp^1$, where $S$ is a K3 surface, forms an unbounded family of Calabi--Yau pairs. 
On the other hand, a result of Koll\'ar and Larsen in~\cite{KL} implies that the base of an elliptic Calabi--Yau manifold $Y$ is rationally connected when $Y$ is not a product.
A conjecture of M\textsuperscript{c}Kernan and Prokhorov,~\cite[Conjecture~3.9]{MP}, predicts that the set of rationally connected varieties $X$ admitting a Calabi--Yau pair structure $(X, \Delta)$ is bounded in any fixed dimension if we restrict the type of singularities of the pair, see also~\cite[Conjecture~1.3]{rccy3}.
Our next result offers a first important piece of evidence for this conjecture in dimension up to $4$.

\begin{theorem}
\label{main}
Fix a positive integer $n\leq 4$ and a finite set $I \subset [0,1]$. Then the set $\mathfrak{D}$ of pairs $(X,\Delta)$ such that 
	\begin{enumerate}
		\item $X$ is a projective variety of dimension $n$,
		\item $(X,\Delta)$ is klt with coefficients of $\Delta$ in $I$ and $\Delta \neq 0$,
		\item $K_X+\Delta$ is numerically trivial, and
		\item $(X,\Delta)$ is not of product-type 
	\end{enumerate}
	forms a log birationally bounded family.
\\
More precisely, there exists a bounded family of Calabi--Yau pairs $\mathfrak{D}'$ such that each pair in $\mathfrak{D}$ is isomorphic in codimension $1$ to a pair in $\mathfrak{D}'$.
\end{theorem}

The condition that $(X,\Delta)$ is not of product-type is analogous to that of Theorem~\ref{elliptic} and it is necessary in order to avoid the examples of unbounded Calabi--Yau pairs from the previous paragraph.
We can show that this condition can be characterized by studying possible outcomes of runs of the Minimal Model Program (MMP), as shown in Theorem~\ref{notproduct}. 
Roughly speaking, $(X,\Delta)$ is of product-type if a run of the MMP ends with a special type of generically isotrivial fibration.
Let us also point out that the condition that $(X,\Delta)$ is klt cannot be relaxed as shown in Example \ref{hirzebruch}.   

An immediate corollary of what we have discussed so far is the birational boundedness of bases of elliptic Calabi--Yau manifolds in low dimension. 
We already saw that is a fundamental step in the proof of Theorem~\ref{elliptic}.

\begin{corollary}
\label{cy.fibration}
Let $n\leq 5$ be an integer. 
Then there are finitely many families, up to birational equivalence, of bases $X$ of elliptic fibrations $Y \to X$ such that
\begin{enumerate}
\item 
$Y$ is terminal projective of dimension $n$ and $K_Y \equiv 0$, and
\item 
$Y$ is not of product-type.
\end{enumerate}
\end{corollary} 

Combining the previous results we get effective control of the torsion index for the canonical class $K_Y$ of $Y$.

\begin{corollary}
\label{effective.nonvanishing}
Fix a positive integer $n\leq 5$. 
Then there exists a positive integer $m_0=m_0(n)$ such that if 
\begin{enumerate}
\item $Y$ is a terminal projective of dimension $n$ and $K_Y \equiv 0$,
\item  there exists an elliptic fibration $Y\rightarrow X$, and
\item $Y$ is not of product type,
\end{enumerate}
then $h^0(Y,\mathcal{O}_Y (m(K_Y)))\neq 0$ for any $m$ divisible by $m_0$.
\end{corollary}

Theorem~\ref{main} implies also uniform boundedness of the Picard number and effective non-vanishing of the log plurigenera of Calabi--Yau pairs.

\begin{corollary}
\label{picard}
	Fix a positive integer $n\leq 4$ and a finite set $I \subset [0,1]\cap\qq$. 
	Then there are integers $\rho$ and $m_0$ depending on $n$ and $I$ such that for any $(X,\Delta)$ as in Theorem~\ref{main}, we have $\rho(X)\leq \rho$ and $h^0(X,\oo_X (m(K_X +\Delta)))\neq 0$ for any $m$ divisible by $m_0$.
\end{corollary}

All of the statements illustrated so far follow from the main technical result of this paper, Theorem~\ref{bounded.bases}.
Roughly speaking, this result shows that a family of Calabi--Yau pairs is bounded provided that each pair is endowed with a Mori fibration whose base also varies in a bounded family. 
It is not hard to see, thanks to standard results in the MMP,~\cite{BCHM}, that given a Calabi--Yau pair it is always  possible to pass to a birational one that is endowed with a Mori fibration. 
Hence, in view of such result, the original  boundedness problem for Calabi--Yau pairs is turned into a boundedness problem for the bases of Mori fibrations for which the total space admits a Calabi--Yau pair structure.
The bases of such fibrations are, in turn, Calabi--Yau pairs, although we lack control on the coefficients of the possible boundaries.

Theorem~\ref{bounded.bases} holds in any dimension and hence might constitute a possible inductive step in the generalization of this paper's main results to dimension greater than $5$.
In fact, it is not unreasonable to expect that the approach we carried out to prove Theorem~\ref{main} extends in higher dimension.
Indeed, this holds true automatically, if we assume a well known conjecture in birational geometry: after Birkar's proof of the BAB conjecture, cf. Theorem~\ref{bab.thm}, assuming the Effective Base Point Freeness Conjecture,~\cite[Conjecture~7.13.3]{PS} would immediately imply Theorem~\ref{main} in any dimension.
	Theorem~\ref{bounded.bases} has been recently refined by Birkar,~\cite{birkar.lcy.fibr}, by slightly different techniques to generalize the strategy formulated in this paper and extend Theorem~\ref{main} to higher dimension.

We expect Theorem~\ref{bounded.bases} to have a wide range of applications. 
For example, combining it with the main theorem in~\cite{HX14}, we obtain the following result.

\begin{corollary}
\label{max.var.cor}
	Fix a positive integer $n$ and a finite set $I\subset [0,1]\cap \mathbb{Q}$. 
	Let $\mathfrak{D}$ the set of pairs $(X, \Delta)$ satisfying
	\begin{enumerate}
	\item $(X, \Delta )$ is a projective  klt pair of dimension $n$,
		\item the coefficients of $\Delta$ belong to $I$,
		\item $K_X+\Delta \equiv 0$,
		\item there exists a contraction morphism $f\colon X\to Z$ with $\dim Z < \dim X$, 
		\item $\Delta$ is big over $Z$, and
		\item $f$ has maximal variation.
	\end{enumerate}
Then, there exists a bounded family of Calabi--Yau pairs $\mathfrak{D}'$ such that each pair in $\mathfrak{D}$ is isomorphic in codimension one to a pair in $\mathfrak{D}'$
\end{corollary}

The variation of $f$ is said to be maximal if there is a non-empty open set of $Z$ over which the isomorphism equivalence classes of pairs given by a fiber $F$ and the restriction of $\Delta$ to $F$ are finite.
Intuitively, $f$ has maximal variation if the image of the rational map from the base of the fibration to the moduli space of the fibres $(F,\Delta|_F)$ has maximal dimension. 
In particular, the moduli part in the canonical bundle formula is a big divisor, see \S~\ref{lc.trivial.fibrations} for more details, in particular Remark~\ref{rmk.max.var}.

Let us conclude with a remark. 
In all statements of this Introduction, the assumption on the existence of a finite set in which the coefficients for the divisor $\Delta$ may vary can be weakened to the existence of a set of coefficients satisfying the descending chain condition, see \S~\ref{prelim.sec} for the definition.
This is a now standard reduction in view of the fundamental results of~\cite{HMX14}. 
In particular, all the results of this paper work in this more general setting and their proofs will be carried out accordingly.

\subsection{Strategy of the proof} 
The proof of Theorem~\ref{main} consists of three main steps which we now proceed to summarize. 
The first two steps work in any dimension and only to complete the final one we need to impose the extra condition on the dimension of the Calabi--Yau pairs being at most 4.

The first step of the proof is to show that given a Calabi--Yau pair we can find a suitable birational model that allows us to adopt an inductive strategy of proof. 
This is done by deriving the following structure theorem for Calabi--Yau pairs.
\begin{theorem}
\label{notproduct}
Let $(X,\Delta)$ be a projective klt Calabi--Yau pair not of product-type with $\Delta \neq 0$. 
Then there exists a birational contraction 
\begin{align*}
		\pi \colon X \dashrightarrow X'
\end{align*}
to a $\mathbb{Q}$-factorial projective klt Calabi--Yau pair $(X', \Delta' := \pi_\ast \Delta)$ with $\Delta' \neq 0$ and a tower of morphisms
	\begin{displaymath}
		\xymatrix{
			X'=Y_0 \ar[r]^{p_0} & Y_1 \ar[r]^{p_1} & Y_2 \ar[r]^{p_2} &\dots 
			\ar[r]^{p_{k-1}}& Y_{k}
		}
\end{displaymath}		
\noindent	
such that each morphism $p_i\colon Y_i \to Y_{i+1}$ is a $K_{Y_i}$-Mori fibre space and $Y_k$ is Fano with klt singularities.
\end{theorem}

Once we are in this framework, we proceed to prove that the set of Calabi--Yau pairs endowed with a Mori fibre space structure where the base belongs to a bounded family is also bounded.
That is in fact our second step.

\begin{theorem} 
\label{bounded.bases}
Fix a positive integer $n$, a DCC set $I\subset [0,1]$ and a bounded family $\mathfrak{F}$ of projective varieties. Then the set of pairs $(X, \Delta)$ satisfying
\begin{enumerate}
\item 
$(X, \Delta )$ is a projective klt pair of dimension $n$,
\item 
the coefficients of $\Delta$ belong to $I$,
\item 
$K_X+\Delta \equiv 0$,
\item 
there exists a contraction $f\colon X\to Z$ with $\dim Z < \dim X$
\item 
$-K_X$ is $f$-ample, and
\item 
$Z \in \mathfrak{F}$
\end{enumerate}
forms a bounded family.
\end{theorem}

As $\Delta$ is relatively ample over $Z$, it is natural to consider ample divisors of the form $\delta\Delta + f^\ast H$, $0<\delta\ll 1$, where $H$ is a very ample Cartier divisor with bounded volume, coming from the boundedness of $Z$.
In order to prove boundedness for the pairs $(X, \Delta)$ in the statement of Theorem~\ref{bounded.bases}, by results of Hacon--M$^\textsc{c}$Kernan--Xu, see \S~\ref{boundedness.subsec}-\ref{bound.cyp.sect}, it suffices to show that we can choose a value of $\delta$ for which we can control the singularities of $(X, (1+\delta)\Delta + f^\ast H)$ as well as bound the volume of $\delta\Delta + f^\ast H$ uniformly on all the pairs.
Control on the singularities follows from the fact that we are working with Calabi--Yau pairs and the coefficients of $\Delta$ vary in a DCC set, cf. Theorem~\ref{hmx_1.5.thm} and Corollary~\ref{elc}. 
Instead, the volume estimate is achieved by reducing inductively to estimating the volume of the restriction of $\Delta$ to the general fiber $F$ of $f$, cf. Proposition~\ref{hyperplane}; 
as $F$ is Fano and $(F, \Delta|_F)$ is klt Calabi--Yau with DCC coefficients, the boundedness of the volume of $\Delta|_F$ is just Theorem~\ref{bound}.

It is well understood that Theorem~\ref{bounded.bases} fails without the klt assumption. 
For example, Hirzebruch surfaces give an example of an unbounded family of log canonical Calabi--Yau pairs admitting Mori fibre space structures over $\mathbb{P}^1$. 
Hence, in these hypotheses, we cannot admit singularities worse than klt, as shown in Example~\ref{hirzebruch}. 
In the case of threefolds, a similar boundedness result has been obtained by Jiang,~\cite{Jia15}, under slightly different assumptions.

The final step is to run induction on the number of factors appearing in Theorem~\ref{notproduct}. 
The base case can be done in any dimension using BAB, proved for surfaces by Alexeev,~\cite{Ale94}, and in general by Birkar,~\cite{bir16bab}.

\begin{corollary}
\label{bounded.lcy.mfs.thm}
Fix a DCC set $I\subset [0,1]$ and a positive integer $n$.
Then the set $\mathfrak D$ of pairs $(X, \Delta)$ satisfying
\begin{enumerate}
\item 
$(X, \Delta )$ is a projective klt pair of dimension $n$,
\item 
the coefficients of $\Delta$ belong to $I$,
\item 
$K_X+\Delta \equiv 0$,
\item 
there exists a contraction $f\colon X\to Z$ with $\dim Z < \dim X$
\item 
$-K_X$ is $f$-ample, and
\item 
$Z$ is normal and $-K_Z$ is big,
\end{enumerate}
forms a bounded family.
\end{corollary}

Unfortunately, the final inductive step in the proof of Theorem~\ref{main} can be carried out only in low dimension. 
The main issue here arises when the tower of Mori fibre spaces constructed in Theorem~\ref{notproduct} contains more than one step and the relative dimension of the morphism is $>1$. 
In this case we are not able to conclude that the intermediate varieties appearing in the tower carry a structure of Calabi--Yau pair with DCC coefficients --  a condition which would allow us to run our argument inductively. 
As the Effective Base Point Freeness Conjecture has been verified for morphisms of relative dimension $1$ by Prokhorov-Shokurov,~\cite{PS}, this does not constitute an issue in dimension up to $4$. 

Let us conclude this section by showing that Theorem~\ref{main} cannot be generalized to log canonical pairs with a simple example.

\begin{example}
\label{hirzebruch}
Let $X=\mathbb{F}_e$ be the $e$-th Hirzebruch surface and let $\Delta=\sum_{i\geq 0} \delta_i \Delta_i$ be an effective divisor such that $K_X+\Delta \equiv 0$. 
We can assume that $\Delta_0= C_0$, where $C_0$ is the section with $C_0^2 = -e$. 
By~\cite[Proposition~V.2.20]{Har}, for $i\geq 1$, we have that $\Delta_i \equiv a_i C_0 + b_i f$ with $a_i \geq 0$ and $b_i \geq a_i e$, where $f$ is the generic fibre. 
The Calabi--Yau condition implies that
	\[
		\begin{cases}
			\delta_{0}+\sum_{i\geq 1}\delta_{i}a_{i}=2, \\
			\sum_{i\geq 1}\delta_{i}b_{i}=e+2.
		\end{cases}
	\]
Then, 
\begin{align*}
e+2=\sum_{i\geq 1}\delta_{i}b_{i}\geq \sum_{i\geq 1}\delta_{i}a_{i}e=e(2-\delta_{0}).
\end{align*}
In particular, $(1-\delta_0)e\leq 2$. If we allow $\delta_0$ to be $1$, the family of Hirzebruch surfaces is an unbounded family of Calabi--Yau pairs with lc singularities. 
On the other hand, if we impose that $\delta_0$ is uniformly bounded away from $1$, the family is bounded. 
We remark that we are only requiring that $(X,\Delta)$ is klt along the generic fiber and as we will see later, this is a very important property for our strategy, cf. \S~\ref{boundedness.sect}. 
\end{example}

The paper is organized as follows: 
in \S~\ref{prelim.sec} we introduce the basic definitions and preliminary results that will be used in the rest of the paper;
in \S~\ref{CY.pairs.sect} we prove Theorem~\ref{notproduct};
in \S~\ref{boundedness.sect} we study Mori fibrations with bounded bases and prove those results that are needed for the proof Theorem~\ref{bounded.bases};
finally, the proofs of the results of the paper appear in \S~\ref{sec.proofs}-\ref{sec.cor.proof}.


\subsection*{Acknowledgements}
We would like to thank Paolo Cascini, Johan de Jong, Antonella Grassi, Christopher Hacon, J\'anos Koll\'ar, James M\textsuperscript{c}Kernan, Edward Witten and Chenyang Xu for many valuable conversations, suggestions and comments.
We wish to thank the anonymous referees for many corrections, improvements, and suggestions regarding the paper.


\section{Preliminaries}\label{prelim.sec}

Let $I \subset \mathbb{R}$ be a subset of the real numbers. We say that $I$ satisfies the descending (resp. ascending) 
chain condition, in short DCC (ACC), if any non-increasing (non-decreasing) subsequence of numbers in $I$ is eventually constant.

\subsection{Minimal model program} 
We follow the notation and definitions used in~\cite{KM}. 
Moreover, we will make use of the Minimal Model Program (MMP, in short) mostly for varieties with non pseudo-effective canonical class. 
In this case, the existence and termination of the MMP  has been established by Birkar--Cascini--Hacon--M$^\text{c}$Kernan. 

\begin{theorem}\label{mmp}\cite[Corollary~1.3.3]{BCHM}
Let $(X,\Delta)$ be a $\qq$-factorial klt pair. Assume $K_{X}+\Delta$ is not pseudo-effective. 
Then we may run a $(K_{X}+\Delta)$-MMP $g \colon X \dashrightarrow X'$ that terminates with a Mori fibre space $f\colon X' 	\rightarrow Z$.
\end{theorem}

Let us recall the definition of a Mori fibre space. 

\begin{definition}
Let $(X,\Delta)$ be a klt pair and $f \colon X\rightarrow Z$ be a projective morphism of normal varieties with $\dim Z <\dim(X)$ and $f_\ast \mathcal{O}_X=\mathcal{O}_Z$. 
Then $f$ is a $(K_X+\Delta)$-Mori fibre space if 
	\begin{enumerate}
		\item $X$ is $\qq$-factorial,
		\item $f$ is a primitive contraction, i.e. the relative Picard number $\rho(X/Z)=1$ and 
		\item $-(K_{X}+\Delta)$ is $f$-ample.
	\end{enumerate}
\end{definition}

The conditions that  $\rho(X/Z)=1$ and $X$ is $\mathbb{Q}$-factorial readily imply that also $Z$ is $\mathbb{Q}$-factorial, see~\cite[Proposition~3.36]{KM}.

\subsection{Boundedness of pairs}
\label{boundedness.subsec}

We recall important definitions and results related to the notion of boundedness for log pairs, cf.~\cite[Definition~3.5.1]{HMX14}.

\begin{definition}
\label{def.bound}
A set $\D$ of projective log pairs is bounded if there exists a pair $(Z,B)$, and a projective morphism $f\colon Z \rightarrow T$ with $T$ of finite type, such that any $(X,\Delta) \in \D$ is isomorphic to the fiber $(Z_t, B_t)$ of $f$ for some closed point $t\in T$.
\\	
A set $\D$ of log pairs is log birationally bounded if there exists a pair $(Z,B)$, where the coefficients of $B$ are all $1$, and a projective morphism $f \colon Z \rightarrow T$ with $T$ of finite type, such that for any $(X,\Delta) \in \D$, there exists a closed point $t\in T$ and a birational map $g\colon Z_t \dashrightarrow X$ such that the support of $B_t$ contains the  support of the strict transform of $\Delta$ and any $g$-exceptional divisor.
\end{definition}

To show that a given set of log pairs is log birationally bounded, we will mainly use the following theorem which is a combination of results in \cite{HMX13, HMX14}.
\begin{theorem} 
\cite[Theorem~3.1]{HMX13},~\cite[Theorem~1.3]{HMX14} 
\label{main.hmx}
Fix two positive integers $n$ and $V$ and a DCC set $I\subset [0,1]$. Then the set of pairs $(X,\Delta)$ satisfying
	\begin{enumerate}
		\item $X$ is a projective variety of dimension $n$,
		\item $(X, \Delta)$ is lc, 
		\item the coefficients of $\Delta$ belong to $I$, and
		\item $0<\Vol(K_X+\Delta)\leq V$,
	\end{enumerate}
	is log birationally bounded.
\end{theorem}
\noindent
The reader can find the definition of the volume $\Vol(D)$ of a divisor $D$ in \cite[Definition~2.2.31]{Laz}.

In certain special cases it is possible to deduce boundedness from log birational boundedness. 

\begin{theorem}\cite[Thorem~1.6]{HMX14}\label{hmx_1.6}
	Fix a positive integer $n$ and two positive real numbers $\delta$ and $\epsilon$. 
	Let $\mathfrak{D}$ be a set of log pairs $(X,\Delta)$ such that: 
	\begin{enumerate}
		\item $X$ is a projective variety of dimension $n$,
		\item $K_{X}+\Delta$ is ample,  
		\item the coefficients of $\Delta$ are at least $\delta$, and
		\item the log discrepancy of $(X,\Delta)$ is greater than $\epsilon$.
	\end{enumerate} 
	If $\mathfrak{D}$ is log birationally bounded then $\mathfrak{D}$ is a bounded set of log pairs.
\end{theorem}

\subsection{Boundedness of Calabi-Yau pairs}
\label{bound.cyp.sect} 
In this subsection, we recall some of the results on boundedness of log pairs that we will need in the rest of this paper. 

A first key result in \cite{HMX14} is the following lemma which will play a fundamental role in \S \ref{boundedness.sect}. 

\begin{lemma}
\cite[Lemma~6.1]{HMX14}\label{kltlemma}
Fix a positive integer $n$ and a set $I \subset [0,1]$ satisfying the DCC.
Then there exists a constant $\epsilon=\epsilon(n, I)>0$ satisfying the following property: \newline 
	let $(X,\Delta)$ be a pair such that	
	\begin{enumerate}
		\item $X$ is a projective variety of dimension $n$,
		\item $(X, \Delta)$ is klt, 
		\item the coefficients of $\Delta$ belong to $I$,
		\item $\Delta$ is a big divisor, and 
		\item $K_X+\Delta$ is numerically trivial.
	\end{enumerate}
	If $\Phi$ is a divisor with $K_{X}+\Phi\equiv 0$ and $\Phi \geq (1-\delta)\Delta$ for some $\delta<\epsilon$, then $(X,\Phi)$ is klt. 
\end{lemma}

A generalization of this result appeared in~\cite[Lemma~4.3]{rccy3}.

Lemma~\ref{kltlemma} easily implies that the boundary divisor of any klt Calabi-Yau pair has bounded volume,
once its coefficients belong to a DCC subset of $[0,1]$.

\begin{theorem}
~\cite[Theorem~B]{HMX14}
\label{bound}
	Fix a positive integer $n$ and a set $I \subset [0,1]$ satisfying the DCC. Then there exists a positive integer $M=M(n, I)$ such that if $(X,\Delta)$ is a pair with 
	\begin{enumerate}
		\item $X$ is a projective variety of dimension $n$,
		\item $(X, \Delta)$ is klt, 
		\item the coefficients of $\Delta$ belong to $I$, and
		\item $K_X+\Delta$ is numerically trivial,
	\end{enumerate}
	then $\Vol(X,\Delta)\leq M$.
\end{theorem}

As a consequence of these results, we can even bound the singularities of Calabi-Yau pairs, by showing the coefficients of Calabi-Yau pairs actually belong to a finite set.

\begin{theorem}~\cite[Theorem~1.5]{HMX14}\label{hmx_1.5.thm}
	Fix a positive integer $n$ and a set $I \subset [0,1]$ satisfying 
	the DCC. Then there is a finite subset $I_0 \subseteq I$ with the following properties: \newline
	If $(X, \Delta)$ is a log pair such that
	\begin{enumerate}
		\item $X$ is a projective variety of dimension $n$,
		\item $(X, \Delta)$ is log canonical, 
		\item the coefficients of $\Delta$ belong to $I$, and
		\item $K_X+\Delta$ is numerically trivial,
	\end{enumerate}
	then the coefficients of $\Delta$ belong to $I_0$.
\end{theorem}

We will need the following well known generalization of the above theorem, see, for example,~\cite[Lemma~2.48]{bab16a}. 

\begin{corollary}
\label{elc}
Fix a positive integer $n$ and a set $I \subset [0,1]$ satisfying the DCC. 
Then there exists a positive real number $\epsilon=\epsilon(n, I)$ such that any $n$-dimensional projective klt Calabi--Yau pair $(X,\Delta)$ is $\epsilon$-lc. 
\end{corollary}

Imposing more conditions on the set of pairs, it is possible to deduce stronger boundedness statements. 
For example, in~\cite{HMX14}, the following result is proved.

\begin{theorem}
\cite[Theorem~1.7]{HMX14}
\label{hmx_1.7.thm}
	Fix a positive integer n, a real number $\epsilon>0$ and a set $I \subset [0,1]$ satisfying the DCC. 
	Let $\D$ be the set of all pairs $(X, \Delta)$ such that
	\begin{enumerate}
		\item $X$ is a projective variety of dimension $n$,
		\item the coefficients of $\Delta$ belong to $I$,
		\item the total log discrepancy of $(X, \Delta)$ is greater than $\epsilon$,
		\item $K_X+\Delta$ is numerically trivial, and
		\item $-K_X$ is ample.
	\end{enumerate}
	Then $\D$ forms a bounded family.
\end{theorem}

In a similar fashion, Hacon-Xu proved the following result.

\begin{theorem}
\cite[Theorem~1.3]{HX14}
\label{hx.thm}
	Fix a positive integer n and a DCC set $I \subset [0,1]\cap\qq$. Let $\D$ be the set of all pairs $(X, \Delta)$ such that
	\begin{enumerate}
		\item $X$ is a projective variety of dimension $n$,
		\item $(X,\Delta)$ is klt with coefficients of $\Delta$ in $I$,
		\item $K_X+\Delta$ is numerically trivial, and
		\item $\Delta$ is big.
	\end{enumerate}
	Then $\D$ forms a bounded family.
\end{theorem}

The above theorems can be thought as a first step towards results predicted by the famous BAB conjecture, which has recently been solved by Birkar.

\begin{theorem}
\cite[Corollary~1.2]{bir16bab}
\label{bab.thm}
		Fix a positive integer n and a positive real number $\epsilon$.
	Then the set of $n$-dimensional projective varieties $X$ such that:
	\begin{enumerate}
	\item there exists an effective divisor $\Delta$ such that the pair $(X, \Delta)$ is an $\epsilon$-lc Calabi--Yau pair, and 
	\item $\Delta$ is big
	\end{enumerate}
forms a bounded family.
\end{theorem}

In dimension $2$, Alexeev proved a more general statement, which we will use in Section~\ref{sec.proofs}.

\begin{theorem}
\cite[Theorem~6.8]{Ale94}
\label{alexeev.thm}
Let $\epsilon$ be a positive real number. 
Let $\mathfrak D$ be the set of projective surfaces $X$ such that there exists an $\rr$-divisor $\Delta$ on $X$ with $(X,\Delta)$ $\epsilon$-klt and $-(K_X + \Delta)$ is nef, where we exclude those surfaces $X$ for which at the same time $K_X$ is numerically trivial and $X$ has at worst Du Val singularities. 
Then $\mathfrak D$ is bounded.
\end{theorem}

\subsection{Generalised pairs}
\label{sect.gen.pair} 
For the definition of b-divisor and related notions, we refer the reader to~\cite{bir.zhang}.
There, the authors introduced also the notion of generalised pairs. 
Let us recall that a b-$\mathbb{R}$-divisor $\mathbf{N}$ is said to descend to the divisor $N'$ on a model $X'$ if $\mathbf{N}$ equals the Cartier closure of its trace $\mathbf{N}_{X'}$ on $X'$ and $\mathbf{N}_{X'}=N'$.
\begin{definition}
Let $T$ be a variety. 
A generalised polarised pair over $T$ is a tuple $(X' \to X, B, M')$ consisting of the following data:
\begin{enumerate}
\item a normal variety $X \to T$ projective over $T$ equipped with a projective birational morphism $\phi \colon X' \rightarrow X$, 
\item an effective $\mathbb{R}$-Weil divisor $B$ on $X$,
\item a b-$\mathbb{R}$-Cartier b-divisor $\mathbf{M}$ over $X$ which descends on $X'$ such that $M':=\mathbf{M}_{X'}$ is nef over $T$, and
\item $K_X +B+ M$ is $\mathbb R$-Cartier, where $M := \phi_\ast M'$.
\end{enumerate}
\end{definition}

When no confusion arises, we refer to the pair by saying that $(X, B+M)$ is a generalised pair with data $X' \rightarrow X \rightarrow T$ and $M'$. 
We call $M'$ the nef part of the generalised pair.

Similarly to log pairs, we can define discrepancies and singularities for generalised pairs, cf.~\cite[Definition~4.1]{bir.zhang}. 

\subsection{Lc-trivial fibrations}
\label{lc.trivial.fibrations}

Since in this paper we deal mostly with Mori fibrations of Calabi-Yau pairs, in this subsection we collect some specific results about lc-trivial fibrations which can be applied to our setting. 

An lc-trivial fibration is a contraction $f \colon X\rightarrow Z$ of normal varieties, with $\dim Z >0$, for which there is an effective divisor $\Delta$ on $X$ such that the pair $(X,\Delta)$ is klt over the generic point of $Z$ and $K_{X}+\Delta \sim_{f,\rr} 0$. 
We will denote this by saying that $f \colon (X, \Delta) \to Z$ is an lc-trivial fibration.
For more details, we refer to~\cite{Amb04, Amb05}, and to~\cite[Theorem~3.1]{MR2944479} for the case of real coefficients.   
The main result which we will need is the following theorem, also known as the canonical bundle formula.

\begin{theorem}
\label{bundle}
Let $f \colon (X, \Delta)\rightarrow Z$ be an lc-trivial fibration.
Then there exist an effective divisor $B_{Z}$ on $Z$ and a b-divisor $\mathbf M$ descending to a nef divisor $M':=\mathbf M_{Z'}$ on a higher model $Z' \to Z$
such that the tuple $(Z' \to Z, B_Z, \mathbf M)$ is a generalised pair.
Moreover, if $(X,\Delta)$ is klt then there exists an effective divisor $0 \leq \overline{M} \sim_\rr M_Z$ such that $(Z,B_{Z} +\overline{M})$ is klt.
\end{theorem}

In the rest of the paper we will refer to $B_Z$ (resp. to $M_Z$) as the boundary part (resp. the moduli part) of the fibration $f$.
It is actually possible to define divisors $B_{Z'}, M_{Z'}$ on every birational model $Z' \to Z$ over $Z$ that are analogues of $B_Z, M_Z$ and that capture the behaviour of the pullback of the fibration $f$ to $Z'$.
Moreover, the collection of all $B_{Z'}$ (resp. $M_{Z'}$) fit together to form a so-called b-divisor, see~\cite[\S~2]{Amb04} for more details.
\\
We quickly recall the definition of $B_Z$: we set 
\begin{equation}
\label{def.bound.part.eqn}
B_Z := \sum (1-l_D) D,
\end{equation} 
where the sum is taken over every prime divisor $D$ on $Z$, and $l_D$ is the log canonical threshold of $f^*D$ with respect to $(X,\Delta)$ over the generic point of $D$. 

The following lemma shows that in many cases we can control the coefficients of the boundary part.
\begin{lemma}
\label{acc.lemma}
Let $I$ be a DCC set.
Let $(X, \Delta)$ be a pair such that the coefficients of $\Delta$ are contained in $I$.
Let $f \colon (X, \Delta) \to Z$ be an lc-trivial fibration with $K_X+ \Delta\sim_{f, \mathbb{R}}0$.
Then, there exists a DCC set $J$ such that the coefficients of the boundary part $B_Z$ are contained in $J$.
\end{lemma}

\begin{proof}
Given a prime divisor $D$ on $Z$, then the coefficient $\mu_D B_Z$ of $D$ in $B_Z$ satisfies $\mu_D B_Z = 1 -l_D$, where $l_D$ is the log canonical threshold of $f^\ast D$ with respect to $(X,\Delta)$ over the generic point of $D$, see~\eqref{def.bound.part.eqn}. 
As $l_D$ is computed over the generic point of $D$, we can assume that the coefficients of $f^\ast D$ belong to $\mathbb{N}_{>0}$. 
Since also the coefficients of $\Delta$ belong to a DCC set, \cite[Theorem 1.1]{HMX14} implies that there exists an ACC set $J'=J'(\dim X, I, \mathbb{N}_{>0})$ such that $l_D$ belongs to $J'$.
Defining $J:=\{1-t \ \vert \ t \in J'\}$, then $J$ is a DCC set and $\mu_D B_Z$ belongs to $J$.
\end{proof}

It is important for induction purposes to be able to control the singularities of the base of an lc-trivial fibration. 
By a theorem of Birkar, it is possible to do so under some conditions on the singularities
of the generic fibre.

\begin{theorem}
\cite[Theorem~1.4]{Bir16}
\label{birk.thm}
	Let $n$ be a positive integer, $\epsilon>0$ be a real number and let $\mathfrak{D}$ be a bounded set of pairs. 
	Then there is a real number $\delta=\delta(n,\epsilon,\mathfrak{D})>0$ satisfying the following: 
	\\
	let $(X,\Delta)$ be a projective log pair and $f\colon X \rightarrow Z$ be a contraction such that 
	\begin{enumerate}
		\item $(X,\Delta)$ is $\epsilon$-lc of dimension $n$,
		\item $K_X+\Delta \sim_{f, \rr} 0$,
		\item the general fibre $F$ of $f$ is of Fano type, and
		\item $(F,\Delta\vert_F)\in \mathfrak{D}$.
	\end{enumerate}
	Then we can choose an $\rr$-divisor $M_{Z}\geq 0$ representing the moduli part so that $(Z, B_Z+M_Z)$ is $\delta$-lc.
\end{theorem}

\begin{remark}
\label{rmk.bab.sing.base}
Birkar's solution of the BAB conjecture,~\cite{bab16a, bir16bab}, implies that for a fixed positive integer $n$, and fixed positive real numbers $\epsilon$, $\eta$,  the set $\mathfrak{D}_{n, \epsilon, \eta}$ of pairs $(Z, \Gamma)$ such that
\begin{enumerate}
\item[(i)] $(Z,\Gamma)$ is $\epsilon$-lc of dimension $n$,
\item[(ii)] $K_Z+\Gamma \sim_{\rr} 0$, and
\item[(iii)] $Z$ is projective of Fano type and the coefficients of $\Gamma$ are $\geq \eta$
\end{enumerate}
is bounded. 
Hence, Theorem~\ref{birk.thm} holds for $\mathfrak{D}=\mathfrak{D}_{n, \epsilon, \eta}$.
The positive real number $\delta$ whose existence is claimed in the theorem will then depend on $n, \epsilon, \eta$.
\end{remark}

It is conjectured in~\cite{PS} that, after passing to a birational model $Z' \to Z$ of the base, a fixed power of the moduli part $M_{Z'}$, only depending on the coefficients of $\Delta$, is base point free. 
In the same paper, the authors prove the conjecture in the special case where the relative dimension of the lc-trivial fibration is one.

\begin{theorem}
\cite[Theorem~8.1]{PS}
\label{semiample}
Fix two positive integers $n$ and $r$ and a DCC set of rational coefficients $I \subset \mathbb{Q} \cap (0,1)$. \newline
Let $(X, \Delta)$ be a projective klt pair with coefficients in $I$ and assume that there exists an lc-trivial fibration $f\colon X \rightarrow Z$ with $n=\dim X =\dim Z +1$  and $r(K_X+\Delta)|_F \sim 0$, where $F$ is the generic fibre of $f$.\newline
Then there exist a positive integer $m=m(n,r, I)$ and a birational morphism $Z' \rightarrow Z$ with $|mM_{Z'}|$ base point free. 
\end{theorem} 

Both the boundary and moduli part reflect the geometry of the fibration and they can be used to deduce important information about the total space and the variation of the fibres. 
In general, we can use the moduli part to measure the variation of the fibres. 
More precisely, let $f\colon (X,\Delta)\rightarrow Z$ be an lc-trivial fibration. 
By~\cite[Theorem~3.3]{Amb05}, there exists a diagram 
\begin{equation}
\label{diag.ambro.1}
	\xymatrix{
		& (X,\Delta) \ar[d]^f & & (X^!,\Delta^!) \ar[d]_{f^!} \\
		& Z &  \bar{Z}\ar[l]^\tau \ar[r]^\rho  & Z^!	
	} 
\end{equation}
\noindent where $\tau$ is generically finite and surjective, $\rho$ is surjective, $f^!$ is an lc-trivial fibration such that there exists a non-empty open set $U\subseteq \bar{Z}$ over which
\begin{align}
\label{diag.ambro.2}
\xymatrix{
			& (X,\Delta)\times_Z \bar{Z}|_U \ar[rr]^\simeq \ar[dr] & & (X^!,\Delta^!)\times_{Z^!} \bar{Z}|_U \ar[dl] \\
			&  & U  & 	
		} 
\end{align}
where the top horizontal arrow is an isomorphism of pairs;
furthermore, the moduli part $M_{Z^!}$ of $f^!$ is big and nef and $\tau^* (M_Z)=\rho^*(M_{Z^!})$. 
The variation of $f$ is defined to be the dimension of $Z^!$ in~\eqref{diag.ambro.1}.

\begin{remark}
\label{rmk.max.var}
We say that the morphism $f$ has maximal variation if the map $\rho$ in~\eqref{diag.ambro.1} is quasi-finite, i.e., if $Z$, $Z^!$ have the same dimension. 
By the above discussion, if $f$ has maximal variation then $M_{\overline Z}$ is big and the same holds for $M_Z$, since $\tau_\ast M_{\overline{Z}}$ is a multiple of $M_Z$, by construction.
\end{remark}

\begin{definition}
\label{def.isotrivial}
Let $(X,\Delta)$ be a klt pair.
Let $f\colon (X, \Delta)\rightarrow Z$ be an lc-trivial fibration.
\begin{enumerate}
\item[(i)]   
We say that $f$ is generically isotrivial if there exists a non-empty open set $U \subset Z$ over which the pairs $(X_z, \Delta\vert_{X_z})$, $z \in U$, are all isomorphic and $Z^!$ is a point in \eqref{diag.ambro.1}.
\item[(ii)] 
We say that $f$ is strongly generically isotrivial if it is generically isotrivial and moreover we can take the map $\tau$ in the diagram in \eqref{diag.ambro.1} to be \'etale in codimension $1$.
\end{enumerate}
\end{definition}

Being generically isotrivial immediately implies that the moduli part is numerically trivial, in view of the diagram in \eqref{diag.ambro.2}.
On the other hand, the notion of strong generic isotriviality of an lc-trivial fibration is more rigid and can be characterized through the following result which is a simple extension of~\cite[Theorem~4.7]{Amb05}.

\begin{theorem}
\label{ambro}
Let $f \colon (X,\Delta)\rightarrow Z$ be an lc-trivial fibration.
Assume that $B_{Z}=0$ and $M_{Z}\sim_\mathbb{R} 0$. 
We denote by $F$ a general fiber of $f$. 
Then the morphism \mbox{$\tau\colon \overline{Z} \rightarrow Z$} in \eqref{diag.ambro.1} can be taken to be a finite Galois covering 
satisfying the following properties: 
\begin{enumerate}
\item[(i)]
$\tau$ is \'etale in codimension $1$, and
\item[(ii)]
there exists a non-empty open subset $U\subset \overline{Z}$ and an isomorphism 
\[
(X,\Delta)\times_Z \overline{Z}\vert_U \rightarrow (F,\Delta\vert_F) \times \overline{Z}\vert_U
\] 
over $U$.
\end{enumerate}
\end{theorem}

\begin{proof}
If $\Delta$ is a $\mathbb{Q}$-divisor and $M_Z \sim_\mathbb{Q} 0$, then the statement of the theorem is implied by~\cite[Theorem~4.7]{Amb05} and there is nothing prove.
\\
We then assume that $\Delta$ is an $\mathbb{R}$-divisor.
Following the proof of~\cite[Theorem~3.1]{MR2944479}, it is possible to find effective $\mathbb{Q}$-divisors $\Delta_1, \dots, \Delta_k$ such that for any $i$, $(X, \Delta_i)$ is klt over the generic point of $Z$, and there exists $0<r_i<1$, $r_i \in \mathbb{R}$ such that 
\begin{align}
\label{eqn.decomp.q-div}
K_X+\Delta = \sum_{i=1}^k r_i (K_X+\Delta_i), \quad \text{and} \quad \sum_{i=1}^k r_i=1.
\end{align}
Moreover, the divisors $\Delta_i$ can be chosen so that for any $i$, $f \colon (X, \Delta_i) \to Z$ is an lc trivial fibration.
\\
{\bf Claim}. For all $i=1, 2, \dots, k$, $K_X+\Delta_i \sim_\mathbb{Q} f^\ast K_Z$.
\begin{proof}
By Theorem~\ref{bundle}, for all $i=1, 2, \dots, k$, there exist an effective $\mathbb Q$-divisor $B_{Z, i}$ and a pseudoeffective $\mathbb Q$-divisor $M_{Z, i}$ such that  
\begin{align}
\label{cbf.interp.eqn}
K_X+\Delta & \sim_\mathbb{Q} f^\ast (K_Z+B_{Z, i}+M_{Z, i}).
\end{align}
By~\eqref{eqn.decomp.q-div} and~\eqref{cbf.interp.eqn}, then 
\[
f^\ast K_Z \sim_\mathbb{R} 
K_X+\Delta = 
\sum_{i=1}^k r_i (K_X+\Delta_i) \sim_\mathbb{Q} 
f^\ast (K_Z+\sum_{i=1}^k r_i B_{Z, i}+\sum_{i=1}^k r_iM_{Z, i}).
\]
As $\sum_{i=1}^k r_i B_{Z, i}$ is effective and $\sum_{i=1}^k r_iM_{Z, i}$ is pseudoeffective, this implies that for all $i$, $M_{Z, i} \equiv 0 = B_{Z, i}$.
Finally, \cite[Theorem~1.3]{Flo} implies that $M_{Z, i} \sim_\mathbb{Q} 0$, which concludes the proof of the claim.
\end{proof}
Applying ~\cite[Theorem~4.7]{Amb05} to $f \colon (X, \Delta_i) \to Z$, it follows that for any $i$ there exists a finite Galois covering  $\tau_i \colon \overline{Z}_i \to Z$ satisfying the following properties: 
\begin{enumerate}
\item[(i)]
$\tau_i$ is \'etale in codimension $1$, and
\item[(ii)]
there exists a non-empty open subset $U_i\subset \overline{Z}_i$ and an isomorphism 
\[
(X,\Delta_i)\times_Z \overline{Z}_i \vert_{U_i} \rightarrow (F,\Delta_i\vert_F) \times \overline{Z}_i\vert_{U_i}
\] 
over $U_i$.
\end{enumerate}
By passing to the simultaneous Galois closure of the composition of the Galois field extensions induced by the morphisms $\tau_i$, there exists a finite Galois covering  $\tau \colon \overline{Z} \to Z$ such that 
\begin{enumerate}
\item
$\tau$ is \'etale in codimension $1$,
\item 
for any $i$, $\tau$ factors through $\tau_i$, and
\item
there exists a non-empty open subset $U\subset \overline{Z}$ and isomorphisms 
\[
(X,\Delta_i)\times_Z \overline{Z} \vert_{U} \rightarrow (F,\Delta_i\vert_F) \times \overline{Z}\vert_{U}
\] 
over $U$.
\end{enumerate}
By the diagram in property (3) above and by~\eqref{eqn.decomp.q-div}, then it follows that over $U$ the morphism
\[
(X,\Delta)\times_Z \overline{Z} \vert_{U} \rightarrow (F,\Delta\vert_F) \times \overline{Z}\vert_{U},
\] 
is an isomorphism, which concludes the proof.
\end{proof}

\begin{remark}
\label{rmk.bir.isotr}
\begin{enumerate}
\item 
Theorem~\ref{ambro} implies that under the condition of the theorem then $f\colon (X, \Delta)\rightarrow Z$ is a strongly birationally isotrivial lc-trivial fibration, cf. Definition~\ref{def.isotrivial}.
In fact, property (ii) in the statement of Theorem~\ref{ambro} implies that in the diagram~\eqref{diag.ambro.1} $Z^!$ can be taken to be a point, as it suffices to define $(X^!, \Delta^!):=(F, \Delta\vert_F)$.
\item 
It follows from~\eqref{def.bound.part.eqn}-\eqref{diag.ambro.2} that the viceversa of the statement of Theorem~\ref{ambro} holds as well: namely, given a pair $(X, \Delta)$ and an lc-trivial fibration $f \colon (X, \Delta) \to Z$, for which the conditions (i)-(ii) of Theorem~\ref{ambro} hold, then $B_Z=0$, by the definition in~\eqref{def.bound.part.eqn}, and $M_Z\sim_\mathbb{R} 0$, by~\eqref{diag.ambro.2} and condition (ii).
\end{enumerate}
\end{remark}

Theorem~\ref{ambro} motivates the following definition.

\begin{definition}
\label{defproduct}
Let $X$ be a normal variety.
\begin{enumerate}
\item[(i)] 
Let $(X,\Delta)$ be a klt pair.
We say that $(X,\Delta)$ is of product-type if there exists a birational contraction $\pi\colon X \dashrightarrow X'$ such that $(X', \pi_* \Delta)$ admits a strongly generically isotrivial lc-trivial fibration.
\item[(ii)]
We say that $X$ is of product-type if there exists an effective divisor $\Delta$ such that $(X, \Delta)$ is a klt pair which is of product-type.
\end{enumerate}
\end{definition}

 By Theorem~\ref{ambro}, if $(X,\Delta)$ is of product-type then $X$ is birational to a quotient of a product. 

\subsection{Birational transformations of fibered Calabi--Yau manifolds}\label{sect.trans.fib.cy}

In \S \ref{sec.proofs}, we will be interested in modifying the base $X$ of an elliptic fibration $f \colon Y \to X$ by means of a birational contraction $X \dashrightarrow X'$ and show that $X'$ is also the base of a terminal elliptic variety with numerically trivial canonical class.
To this end, we prove the following result.

\begin{proposition}
\label{small.qfact.prop}
Let $Y$ be a projective terminal variety with $K_Y \equiv 0$. 
Let $f \colon Y \to X$ be a contraction of terminal varieties.
Let $s \colon X' \to X$ be a small morphism.
Then there exists a $\mathbb{Q}$-factorial terminal Calabi--Yau variety $Y'$ and a surjective morphism $f' \colon Y' \to X'$.
 In particular $Y'$ is isomorphic to $Y$ in codimension one.
\end{proposition}

\begin{proof}
We denote by $T$ (respectively, $T'$) the exceptional locus of $s^{-1}$ (resp. $s$) on $X$ (resp. $X'$), so that $X \setminus T$ is isomorphic to $X' \setminus T'$.\\

{\bf Claim}. 
{\it
There exists a terminal $\mathbb{Q}$-factorial variety $\overline Y$, a birational contraction $p \colon \overline{Y} \to Y$, and a contraction $\overline{f} \colon \overline{Y} \to X'$ for which the following diagram commutes
 \begin{eqnarray}\label{res.indet}
  \xymatrix{
  Y \ar[d]_{f} & \overline{Y} \ar[l]_{p} \ar[d]^{\overline f} & \\
  X & X'. \ar[l]_{s}
  }
 \end{eqnarray}
Moreover, $\overline{f}(\overline E) \subset T'$, where $\overline E$ is the divisorial part of the exceptional locus of $p$}.
\begin{proof}[Proof of the Claim]
Let $\Gamma \subset Y \times X'$ be normalization of the main component of the Zariski closure  of the graph of the morphism $s^{-1} \circ f \colon Y \dashrightarrow X'$.
We denote by
\[
\xymatrix{
& \Gamma \ar[dr]^{p_2} \ar[dl]_{p_1} &\\
Y & & X'
}
\]
the restrictions of the projections to the two factors of $Y \times X'$.
By construction, $p_1$ is birational and it is an isomorphism over the Zariski open set $X\setminus T \subset X$.
\\
Let $\widetilde Y$ be a resolution of $\Gamma$.
Then we can run the $K_{\widetilde Y}$-MMP relatively over $\Gamma$ with scaling of an ample divisor which terminates with a relatively minimal model $\pi \colon \overline{Y} \to \Gamma$ and $\overline{Y}$ is a terminal $\mathbb{Q}$-factorial projective variety.
We denote by $p:=p_1 \circ \pi$, $\overline f := p_2 \circ \pi$.
As $Y$ is terminal and $p_1$ is an isomorphism over $X \setminus T$, then defining $U := f^{-1}(X \setminus T)$, it follows that $p\vert_{p^{-1}(U)} \colon p^{-1}(U) \to U$ is a small $\mathbb{Q}$-factorial model, by minimality of $\overline{Y}$ over $\Gamma$.
Hence, the center on $Y$ of all $p$-exceptional divisors is contained in $Y \setminus U=f^{-1}(T)$, which then implies the second part of the statement.
\end{proof} 
\noindent
Since $Y$ is terminal, it follows that 
\begin{align}
\label{eqn:terminal.mod}
K_{\overline{Y}}= p^\ast K_Y + E \sim_\mathbb{Q} E,
\end{align} 
where $E$ is an effective divisor whose support coincides with $\overline E$.
By the claim, $E$ is $f'$-exceptional in the sense \cite[Definition 2.9]{Lai}, since $\overline{f}(E)$ has codimension $\geq 2$ in $X'$, as $\overline{f}(E) \subset T'$.
By construction, the general fibre $F'$ of $f'$ has a good minimal model, by~\eqref{eqn:terminal.mod} and since $K_F \sim_\mathbb{Q}0$ on a general fibre $F$ of $f$ -- and $F$ is terminal by adjunction.
Hence, by \cite[Theorem 1.1]{HX}, it follows that there is a good minimal model for $K_{\overline{Y}}$ over $X'$.
Then,~\cite[Corollary 2.9]{HX} implies that any run of the relative $K_{\overline{Y}}$-MMP with scaling of an ample divisor terminates. 
Running one such MMP, we can assume that there exists a minimal model $Y' \to X'$. 
By \cite[Lemma 2.10]{Lai}, as $E$ is $f'$-exceptional, it follows that the relative $K_{\overline{Y}}$-MMP over $X'$ contracts $E$ and hence it terminates with a variety $Y'$ such that $K_{Y'} \equiv 0$ and $Y'$ is terminal.
\\
The final claim in the statement of the proposition follows from the fact that we have contracted all and only the divisors in the support of $E$ whose support coincided with the exceptional locus of $p$.
\end{proof}

A similar result holds if we modify the base of an elliptic fibration by means of a birational contraction.
\begin{proposition}
\label{bir.contr.cy.prop}
Let $Y$ be a projective terminal variety with $K_Y \equiv 0$. 
Let $f \colon Y \to X$ be a contraction of normal varieties. 
Assume that $X$ is a $\mathbb{Q}$-factorial variety and let $t\colon X \dashrightarrow X'$ be a birational contraction to a normal projective variety $X'$.
Then there exists a $\mathbb{Q}$-factorial terminal Calabi--Yau variety $Y'$ and a surjective morphism $f' \colon Y' \to X'$.
In particular $Y'$ is isomorphic to $Y$ in codimension one.
\end{proposition}

\begin{proof}
Let $H'$ be an ample Cartier divisor on $X'$, so that $X'={\rm Proj}\; R(X', H')$, where $R(X', H')$ is the ring of section of $H'$.
Let $H$ be the pullback of $H'$ on $X$\footnote{As $X$ is projective, then the map $t$ is defined over a big open set $U \subset X$ over which the pullback $H_U$ of $H'$ is defined; then, as the complement of $U$ in $X$ has codimension at least $2$, it suffices to take the $H$ to be the Zariski closure of $H_U$.}.
As $t$ is a birational contraction, then $R(X', H') = R(X, H)$.
As $f$ has connected fibers, then $R(X', H') = R(X, H)= R(Y, f^\ast H)$.
\\
If we consider the pair $(Y, \epsilon f^\ast H)$, $0<\epsilon\ll 1$, then $K_Y+f^\ast H$ is abundant since $H$ is big on $X$ and $K_Y \sim_
\mathbb{Q} 0$. 
By~\cite[Theorem 4.3]{GL}, $(Y, \epsilon f^\ast H)$ has a good minimal model, and~\cite[Corollary 2.9]{HX} implies that any run of the $(K_Y+\epsilon f^\ast H)$-MMP with scaling of an ample divisor terminates. 
Hence, running one such MMP we can assume that there exists a minimal model $Y\dashrightarrow Y''$ for $(Y, \epsilon f^\ast H)$.
Denoting by $G''$ the strict transform of $f^\ast H$ on $Y''$, $G''$ is nef on $Y''$, and moreover, by~\cite[Proposition~3.13, Corollary~3.14]{GL}, it is also abundant.
Hence, as $K_Y \sim_
\mathbb{Q} 0$,~\cite[Theorem~1.1]{Fuj11} implies that $G''$ is semiample on $Y''$.
As by construction $R(Y'', G'')= R(Y, f^\ast H) = R(X', H')$, there exists a morphism $Y'' \to X'$.
Passing to a terminalization $Y' \to Y''$ of $(Y'', 0)$ yields the desired model.\\
The final claim follows from \cite[Corollary~3.54]{KM}, as $Y'$ is also terminal projective with $K_{Y'} \equiv 0$ since in going from $Y''$ to $Y'$ we only extracted divisors of discrepancy $0$ with respect to $K_{Y''}$.
\end{proof}

We terminate this section by proving a result that illustrates a criterion to determine when a projective terminal variety with trivial canonical class is of product type.

\begin{lemma}
\label{lem:cy.not.prod.type}
Let $Y$ be a terminal projective variety with $K_Y \equiv 0$.
Let $f \colon Y \to Z$ be a contraction of normal projective varieties.
\begin{enumerate}
\item 
If $0< \dim Z < \dim Y$ and $K_Z \equiv 0$, then $Y$ is of product type.
\item 
If $Z$ is of product type, then $Y$ is also of product type.
\end{enumerate}
\end{lemma}

\begin{proof}
\begin{enumerate}
\item 
As $K_Y\equiv  0$, the canonical bundle formula  $K_Y \sim_\mathbb{Q} 0$.
Hence, by the canonical bundle formula, Theorem~\ref{bundle}, $B_Z=0$ and $M_Z\sim_\mathbb{Q} 0$. 
But this implies that $f$ is strongly generically isotrivial by Theorem~\ref{ambro}, see Remark~\ref{rmk.bir.isotr}, and $Y$ is of product type, by Definition~\ref{defproduct}.

\item
As $Z$ is of product type, then there exists a boundary $\Delta$ such that $(Z, \Delta)$ is of product type.
Hence, there exists a birational contraction $Z \dashrightarrow Z'$ and a contraction $Z' \to W'$ which is strongly generically isotrivial, by Definition~\ref{defproduct}.
In particular, $K_{W'} \sim_\mathbb{Q} 0$.
\\
By Proposition~\ref{small.qfact.prop}, we can assume that $Z$ is $\mathbb{Q}$-factorial. 
Hence, by Proposition~\ref{bir.contr.cy.prop} there exists a birational terminal projective variety $Y'$ with $K_{Y'}\equiv 0$, and a contraction $f \colon Y' \to W'$.
But then part (1) implies that $Y'$ is product type and the same holds for $Y$ since it is isomorphic in codimension one to $Y'$.
\end{enumerate}
\end{proof}


\section{A structure theorem for Calabi-Yau pairs}
\label{CY.pairs.sect}

We would like to characterize Calabi-Yau pairs which behave like a product of lower dimensional varieties. 
In this section we prove a structure theorem for Calabi-Yau pairs which is inspired by the following theorem 
of Koll\'ar-Larsen.

\begin{theorem}\cite[Theorem 3]{KL}
	Let $X$ be a smooth, simply connected projective Calabi-Yau variety and 
	$f \colon X\dashrightarrow Y$ a dominant map such that $Y$ is not uniruled. Then one can write 
	\[
		\xymatrix{
			X \ar[r]^{\pi} & X_1 \ar@{-->}[r]^{g} & Y
		},
	\]
	where $\pi$ is a projection to a direct factor of $X\cong X_1 \times X_2$ and $g\colon X_1 \dashrightarrow Y$ is generically finite.
\end{theorem}

Any run of a $K_X$-MMP with scaling of an ample divisor on a Calabi-Yau pair $(X, \Delta)$, with $\Delta > 0$, terminates with a Mori fibre space structure.
If the base of this fibration is a variety with trivial canonical class then the pair $(X, \Delta)$ is of product-type  by Definition~\ref{defproduct}. 
If that is not the case, we can repeat the same procedure on the base of the Mori fibre space.

By iterating this procedure, we can obtain the following description of Calabi-Yau pairs which is 
the main result of this section.

\begin{theorem}
\label{structure}
Let $(X, \Delta)$ be a projective klt Calabi-Yau pair with $\Delta \neq 0$. 
Then there exists a birational contraction 
\[
		\pi \colon X \dashrightarrow X'
\]
to a projective Calabi-Yau pair $(X', \Delta' =\pi_\ast \Delta)$, $\Delta'\neq 0$ and a tower of morphisms
\begin{align}\label{mfs.tower.diag}
\xymatrix{
	X'=Y_0 \ar[r]^{p_0} & Y_1 \ar[r]^{p_1} & Y_2 \ar[r]^{p_2} &\dots \ar[r]^{p_{k-1}}& Y_{k}
}
\end{align}		
\noindent such that 
\begin{enumerate}
\item[(i)] 
for any $i$, $Y_i$ is $\mathbb{Q}$-factorial,
\item[(ii)] 
for any $1 \leq i < k$, there exists a boundary $\Delta_i \neq 0$ on $Y_i$ such that $(Y_i, \Delta_i)$ is a projective klt Calabi-Yau pair,
\item[(iii)] 
for any $0 \leq i \leq k$ the morphism $p_i\colon Y_i \to Y_{i+1}$ is a $K_{Y_i}$-Mori fibre space, in particular $\rho(Y_i/Y_{i+1})=1$, and
\item[(iv)]  
$Y_k$ is a projective klt Calabi-Yau variety, i.e., $K_{Y_k} \equiv 0$ and $Y_k$ has klt singularities.
\end{enumerate}		
\noindent	
If $\dim Y_k >0$, then $(X, \Delta)$ is of product-type.\\
Moreover, if $X$ is $\mathbb{Q}$-factorial, then the birational contraction $X \dashrightarrow X'$ is a composition of divisorial contractions and $(K_X+\Delta)$-flops.
\end{theorem}

\begin{proof}
We prove the statement of the theorem by induction on the dimension of $X$, which we denote by $n$.\\
When $n=1$, then $X=\mathbb{P}^1$ and it suffices to consider the constant morphism to a point.
Let us now assume that $n>1$ and that the theorem holds for any pair of dimension $<n$ satisfying the hypotheses of the statement.
\\
By~\cite[Corollary~1.37]{Kol13}, we can replace $(X,\Delta)$ with a small $\qq$-factorial model without losing the klt Calabi-Yau condition; it suffices to prove the result for this new pair, as it is isomorphic in codimension one to $(X, \Delta)$.
Abusing notation, we denote by $(X, \Delta)$ this new model, so that, we can assume that the pair $(X, \Delta)$ is $\mathbb{Q}$-factorial.
Since $\Delta \neq 0$, $K_{X}$ is not pseudo-effective and, by Theorem \ref{mmp}, we can run a $K_X$-MMP with scaling of an ample divisor
\begin{align}
\label{diag.3.2}
		\xymatrix{
			X \ar@{-->}[r]^{\pi} & \overline{X} \ar[r]^{\overline{m}} & Y,
		}
\end{align}		
which terminates with a $K_{\overline{X}}$-Mori fibre space $\overline{m} \colon \overline{X} \to Y$.
\newline
Let us note that since $X$ is $\mathbb{Q}$-factorial, then the same holds for $\overline{X}$ and $Y$ as they are outcomes of a run of the MMP.
Furthermore, $\overline{\Delta}=\pi_\ast \Delta \neq 0$ because we are running a $K_X$-MMP and $K_X +\Delta \sim_\mathbb{R} 0$.
Hence, $K_{\overline{X}}+\overline{\Delta} \sim_\mathbb{R} 0$ and $(\overline{X}, \overline{\Delta})$ is klt, and Theorem~\ref{bundle} implies that there exist an effective divisor $\Gamma$ on $Y$ such that $(Y, \Gamma)$ is klt and  $K_{Y} +\Gamma \sim_\mathbb{R} 0$.
\newline
If $\Gamma =0$ then $K_{Y} \equiv 0$ and we are done, by taking $Y_0=\overline{X}$, $Y_1=Y$. 	
Otherwise, $\Gamma \neq 0$ and by inductive hypothesis there exists a diagram
\begin{align*}
\xymatrix{
Y \ar@{-->}[r]^{q_{0}} & Y_{1} \ar[r]^{p_1} & Y_2 \ar[r]^{p_2} &\dots \ar[r]^{p_{k-1}}& Y_{k}
},
\end{align*}
such that $q_0$ is a birational contraction and the tower of morphisms 
\begin{align*}
\xymatrix{
Y_{1} \ar[r]^{p_1} & Y_2 \ar[r]^{p_2} &\dots \ar[r]^{p_{k-1}}& Y_{k}
}
\end{align*} 
satisfies conditions $(i)-(iv)$ in the statement of the theorem.
\\
The statement of the theorem then follows if we show that there exist  a $\mathbb{Q}$-factorial klt Calabi--Yau pair $(X', \Delta')$ with $\Delta' \neq 0$, a birational contraction $\pi' \colon \overline{X} \dashrightarrow X'$ which is a composition of divisorial contractions and $(K_{\overline{X}}+\overline{\Delta})$-flops, and a $K_{X'}$-Mori fiber space $m' \colon X' \to Y_1$ such that the following diagram commutes
\begin{align}
\label{long.diag}
\xymatrix{
X \ar@{-->}[r]^{\pi} & \overline{X} \ar[d]^{m'} \ar@{-->}[r]^{\pi'} & X'=:Y_0 \ar[d]^{m' =:p_0}\\ 
& Y \ar@{-->}[r]^{q_{0}} & Y_{1} \ar[d]^{p_1} \\
& & Y_2 \ar[d]^{p_2} \\
& & \vdots \ar[d]^{p_{k-1}}\\
& & Y_{k}.
}
\end{align}
By inductive hypothesis, as $Y$ is $\mathbb{Q}$-factorial, the birational contraction $q_0$ can be decomposed as a sequence of divisorial contractions and $(K_Y+\Gamma)$-flops
\begin{align*}
\xymatrix{
Y = Y^{(0)}\ar@/_1pc/@{-->}[rrrrr]_{q_{0}} \ar@{-->}[r]^{s_{1}} & Y^{(1)} \ar@{-->}[r]^{s_{2}} & Y^{(2)} \ar@{-->}[r]^{s_{3}} & \dots \ar@{-->}[r]^{s_{j-1}} & Y^{(j-1)} \ar@{-->}[r]^{s_{j}} & Y^{(j)}=Y_{1} \\
}.
\end{align*}
Thus, it suffices to prove the following claim.
We will use the following notation $(X^{(0)}, \Delta^{(0)}):=(\overline{X}, \overline{\Delta})$.\\

{\bf Claim.}
{\it 
For any $0\leq t \leq j$, there exists a $\mathbb{Q}$-factorial klt Calabi--Yau pair $(X^{(t)}, \Delta^{(t)})$, with $\Delta^{(t)} \neq 0$, and a $K_{X^{(t)}}$-Mori fibre space $m'_t \colon X^{(t)} \to Y^{(t)}$ such that the following diagram commutes
\begin{align*}
\xymatrix{
(X^{(0)}, \Delta^{(0)})\ar@/^2pc/@{-->}[rrr]^{r_j \circ r_{j-1}\circ \dots \circ r_2 \circ r_1=:\pi'} \ar@{-->}[r]^{r_{1}} \ar[d]^{m'=:m'_0}& 
(X^{(1)}, \Delta^{(1)}) \ar@{-->}[r]^{r_{2}} \ar[d]^{m'_1}& 
\dots \ar@{-->}[r]^{r_{j}} & 
(X^{(j)}, \Delta^{(j)})=Y_{0} \ar[d]^{m'_{j}=:p_0} \\
Y = Y^{(0)}\ar@/_1pc/@{-->}[rrr]_{q_{0}} \ar@{-->}[r]^{s_{1}} & 
Y^{(1)} \ar@{-->}[r]^{s_{2}} & 
\dots \ar@{-->}[r]^{s_{j}} & 
Y^{(j)}=Y_{1} \\
},
\end{align*}
$\Delta^{(t)}:=r_{t\ast} \Delta^{(t-1)}$, and $r_t$ is a birational contraction that is a composition of divisorial contractions and $(K_{X^{(t-1)}}+\Delta^{(t-1)})$-flops.
}
\begin{proof}[Proof of the Claim]
We use induction on $t \in \{0, 1, \dots, j\}$.\\
For $t=0$, there is nothing to prove.
Thus, we can assume that the statement of the claim holds up to $t-1$ and we shall prove that it holds for $t$ as well.
By inductive hypothesis, we can assume that the following diagram
\begin{align*}
\xymatrix{
X^{(0)}, \Delta^{(0)}) \ar@{-->}[r]^{r_{1}} \ar[d]^{m'_0}& 
(X^{(1)}, \Delta^{(1)}) \ar@{-->}[r]^{r_{2}} \ar[d]^{m'_1}& 
\dots \ar@{-->}[r]^{r_{t-1}} & 
(X^{(t-1)}, \Delta^{(t-1)}) \ar[d]^{m'_{t-1}}\\
Y = Y^{(0)} \ar@{-->}[r]^{s_{1}} & 
Y^{(1)} \ar@{-->}[r]^{s_{2}} & 
\dots \ar@{-->}[r]^{s_{t-1}} & 
Y^{(t-1)}
}
\end{align*}
exists and satisfies the conditions in the statement of the claim.
\\
Hence it suffices to show that the diagram 
\begin{align*}
\xymatrix{
(X^{(t-1)}, \Delta^{(t-1)}) \ar[d]^{m'_{t-1}} &
\\
Y = Y^{(t-1)} \ar@{-->}[r]^{s_{t}} & 
Y^{(t)}
}
\end{align*}
can be completed to a diagram 
\begin{align}
\label{short.diag}
\xymatrix{
(X^{(t-1)}, \Delta^{(t-1)}) \ar[d]^{m'_{t-1}} \ar@{-->}[r]^{r_t} &
(X^{(t)}, \Delta^{(t)}) \ar[d]^{m'_{t}}
\\
Y = Y^{(t-1)} \ar@{-->}[r]^{s_{t}} & 
Y^{(t)}
},
\end{align}
where $r_t$ is a composition of divisorial contractions and $(K_{X^{(t-1)}}+\Delta^{(t-1)})$-flops and $m'_t$ is a $K_{X^{(t)}}$-Mori fibre space.
Since $s_t$ is either a divisorial contraction or a $(K_{Y^{(t-1)}} + \Gamma^{(t-1)})$-flop by inductive hypothesis, where $\Gamma^{(t-1)}$  is the strict transform of $\Gamma$ on $Y^{(t-1)}$, then the existence of the diagram in \eqref{short.diag} follows from Lemma \ref{contraction.step} and \ref{flip.step}.
\end{proof}
\noindent
If $\dim Y_k >0$ then Definition \ref{defproduct} immediately implies that $(X, \Delta)$ is of product type: in fact, as $K_{Y^k} \sim_\mathbb{Q} 0$, then the lc-fibration $(X', \Delta') \to Y^k$ that we have constructed is strongly birationally isotrivial by Theorem~\ref{ambro}, see Remark~\ref{rmk.bir.isotr}, and $\pi' \circ \pi$ is a birational contraction.
\\
If $X$ is $\mathbb{Q}$-factorial, then the birational contraction $\pi' \circ \pi \colon X \dashrightarrow X'$ is a composition of divisorial contraction and $(K_X+\Delta)$-flops. 
The map $\pi \colon X \dashrightarrow \overline{X}$ in \eqref{diag.3.2} is constructed by running a $K_X$-MMP, hence it is a composition of divisorial contractions and $K_X$-flips; as $K_X+\Delta \sim_\mathbb{R} 0$, those $K_X$-flips are also $(K_X+\Delta)$-flops.
On the other hand, $\pi' \colon \overline{X} \to Y_0$ in \eqref{long.diag} is in turn a composition of divisorial contractions and $(K_{\overline{X}}+\overline{\Delta})$-flops, as shown in the proof of the Claim.
As $\pi$ is crepant with respect to $(X, \Delta)$, then those $(K_{\overline{X}}+\overline{\Delta})$-flops are also $(K_{X}+\Delta)$-flops.
\end{proof}

In the previous proof we used the following two lemmas.

\begin{lemma}
\label{contraction.step}
Let $(X,\Delta)$ be a projective klt Calabi-Yau pair with $\Delta\neq 0$. 
Let $p \colon X\rightarrow Y$ be a $K_X$-Mori fibre space of $\mathbb{Q}$-factorial varieties. 
Let $q \colon Y \rightarrow Y'$ be a divisorial contraction. 
Then there exist a $\mathbb{Q}$-factorial projective klt Calabi-Yau pair $(X',\Delta')$, with $\Delta' \neq 0$, a $K_{X'}$-Mori fibre space $p' \colon X'\rightarrow Y'$, and a birational contraction $f\colon X\dashrightarrow X'$ which is a composition of a finite number of $(K_X+\Delta)$-flops followed by a divisorial contraction, such that the following diagram commutes
\begin{displaymath}
	\xymatrix{
		& X \ar@{-->}[r]^{f} \ar[d]_{p} &  X' \ar[d]^{p'}\\
		& Y \ar[r]^{q}  &  Y'	.
	}
\end{displaymath}
\end{lemma}

\begin{proof}
Since $p$ is a $K_X$-Mori fibre space, we have that $\dim \NS(X/Y')=2$. 
Let $F$ be the irreducible divisor contracted by $q$ on $Y$ and let $E=p^\ast(F)$.
As for $0<\epsilon\ll 1$ $(X, \Delta+\epsilon E)$ is klt, $K_X+\Delta + \epsilon E\sim_\mathbb{R} \epsilon E$, and $\Delta$ is relatively big, we can run the $E$-MMP with scaling relatively over $Y'$ and this must terminate.
But then, as $E$ is effective and $q(p(E))$ has codimension $\geq 2$ in $Y'$, \cite[Lemma 2.10]{Lai} implies that this run of the relative $E$-MMP must terminate by contracting $E$ to a $\mathbb{Q}$-factorial model $X'$ of $X$ and $E$ is irreducible since $\rho(X'/Y')=1$.
Hence, the birational contraction $f \colon X \to X'$ is a composition of $E$-flips, followed by a divisorial contraction.
As $K_X+\Delta \sim_\mathbb{R} 0$, those $E$-flips are also $(K_X+\Delta)$-flops.
\\
Finally, the claim that $\Delta' \neq 0$ follows because $\Delta$ is relatively big over $Y'$, hence, $0\neq f_\ast \Delta =\Delta'$.
\end{proof}

\begin{lemma}
\label{flip.step}
Let $(X,\Delta)$ be a projective klt Calabi-Yau pair with $\Delta\neq 0$.  
Let $p \colon X\rightarrow Y$ be a $K_X$-Mori fibre space of $\mathbb{Q}$-factorial varieties with $K_X+\Delta \sim_{\mathbb{R}} p^\ast(K_Y+G)$, and $(Y, G)$ is a klt pair.
Let $q \colon Y \dashrightarrow Y'$ be a $(K_Y+G)$-flop. 
Then there exists a projective klt Calabi-Yau pair $(X',\Delta')$, with $\Delta'\neq 0$, a $K_{X'}$-Mori fibre space $p' : X'\rightarrow Y'$, and an isomorphism in codimension one $f\colon X\dashrightarrow X'$ such that the following diagram commutes
\begin{displaymath}
	\xymatrix{
		& X \ar@{-->}[r]^{f} \ar[d]_{p} &  X' \ar[d]^{p'}\\
		& Y \ar@{-->}[r]^{q}  &  Y'	
	}
\end{displaymath}
\end{lemma}

\begin{proof}
Let $\pi \colon Y\rightarrow Z$, $\pi' \colon Y'\rightarrow Z$ be the flopping contractions associated to $q$ and let $H$ be an effective relatively anti-ample Cartier divisor on $Y$ over $Z$.
We denote $\overline{H}:=p^\ast H$.
We want to produce a map $f$ which is an isomorphism in codimension one and that makes the following diagram commute: 
\begin{displaymath}
	\xymatrix{
		& X \ar[dl]_{p} \ar@{-->}[rr]^{f} & & X' \ar[dr]^{p'} & \\
		Y \ar[drr]^{\pi} \ar@{-->}[rrrr]^{q}& & & & Y' \ar[dll]_{\pi'}\\
	 	& & Z &	&	
	}
\end{displaymath}	
As for $0<\epsilon\ll 1$ $(X, \Delta+\epsilon \overline{H})$ is klt, $K_X+\Delta + \epsilon \overline{H}\sim_\mathbb{R} \epsilon \overline{H}$, and $\Delta$ is relatively big, we can run the $\overline{H}$-MMP with scaling of an ample divisor relatively over $Z$ and this must terminate with a $\mathbb{Q}$-factorial relatively minimal model 
\[
\xymatrix{
X \ar[dr]_{\pi \circ p} \ar@{-->}[rr]^f & & X' \ar[dl]^{s'}\\
& Z &
}
\]
\noindent
The strict transform $\overline{H}'$ of $\overline{H}$ on $X'$ is semiample since $\Delta'$ is big over $Z$, where $\Delta':=f_\ast \Delta$.
As 
\[
Y'={\rm Proj}_{\mathcal{O}_{Z}}(\oplus_{n \geq 0} \pi_\ast \mathcal{O}_Y(nH))= {\rm Proj}_{\mathcal{O}_{Z}}(\oplus_{n \geq 0} (\pi \circ p)_\ast \mathcal{O}_X(n\overline{H})),
\] 
\noindent
Then the relative Iitaka fibration of $\overline{H}'$ over $Z$ induces a morphism $p' \colon X' \to Y'$, since, for any $n \geq 0$, 
\[
\pi_\ast \mathcal{O}_Y(nH) = (\pi \circ p)_\ast \mathcal{O}_X(n\overline{H})) = s'_\ast \mathcal{O}_{X'}(n\overline{H}'),
\]
\noindent
as $f$ is constructed as the run of a $\overline{H}$-MMP and $f_\ast \overline{H}=\overline{H}'$.
\\
We have that $\rho(X'/Z) \leq 2$ since $f$ is a birational contraction and $q$ is an isomorphism in codimension one of $\mathbb{Q}$-factorial varieties.
Since also $\rho(Y'/Z)=1\leq \rho(X'/Y')$, then $\rho(X'/Z)=2$; 
moreover, as $X'$ is $\mathbb{Q}$-factorial then $X$, $X'$ are isomorphic in codimension one, again, as $f$ by construction is a birational contraction and $\rho(X'/Z)= \rho(X/Z)$.
\\
Finally the claim that $\Delta' \neq 0$ follows because $\Delta$ is relatively big over $Y'$, hence, $0<f_\ast \Delta =\Delta'$.
\end{proof}

We conclude this section with a remark on the differences between the property of being rational connected and the property of having product type.
\begin{remark}\label{rat.connected}
If $(X,\Delta)$ is a projective klt Calabi-Yau pair not of product-type then Theorem \ref{structure} proves that $X$ is rationally connected, as $X$ is decomposed in a tower of morphisms with rationally connected fibers.
On the other hand, if $(X,\Delta)$ is of product-type and $X$ is not rationally connected, then Theorem \ref{structure} provides the MRC fibration.
Let us point out that if $(X,\Delta)$ is of product-type it is not always the case that $X$ is not rationally connected: there are examples of rationally connected varieties $Y$ with klt singularities and $K_Y\equiv 0$, see \cite{MR3329200}.
\end{remark}

As we saw in the Introduction, log birational boundedness of Calabi-Yau pairs does not hold if we allow ourselves to consider also product-type pairs.
Nonetheless, it is possible that the boundedness can still be proven if we consider $n$-dimensional product-type pairs $(X, \Delta)$ with $X$ rationally connected and the coefficients of $\Delta$ vary in a fixed DCC set, see~\cite[Conjecture~1.3]{rccy3}.
A proof of this fact appears in~\cite[Theorem~1.4]{bds}.


\section{Boundedness of Mori fibre spaces}
\label{boundedness.sect}

In this section we prove the main technical result needed for Theorem \ref{main}. 
More precisely, we show that if we fix a bounded set of varieties $\D$, then the set of klt Calabi-Yau pairs endowed with a Mori fibre space structure mapping to one of the varieties in $\D$ is bounded as well. 
The strategy is to lift a very ample divisor with bounded volume from the base to the total space, suitably perturb it by means of the boundary of the klt Calabi--Yau pair and then use the results from Section \ref{boundedness.subsec} to show that boundedness must holds for the Calabi--Yau pairs.

\subsection{Boundedness of volume for Mori fibre spaces of generalized pairs}

The main aim of this subsection is to prove the Theorem~\ref{log.birat.bnd.gen.pairs.thm} which represents a first step towards proving the boundedness of Calabi-Yau pairs endowed with a Mori fibre space whose base belongs to a bounded family.
Here, we will work in the more comprehensive setting of generalized pairs.
For the definition and basic properties of generalized pairs, the interested reader can check~\cite[\S 4]{bir.zhang}.

We start by first recalling the following theorem of Birkar which 
is one of the fundamental tools in his proof of \ref{bab.thm}.

\begin{theorem}
\label{bnd.lct.thm}~\cite[Theorem~1.6]{bir16bab}
Let $n,r$ be positive integers and $\epsilon$ a positive real number. 
Then  there exists a positive real number $t= t(n,r,\epsilon)$ satisfying the following. 
	Assume that
	\begin{itemize}
	 \item $(F,B)$ is a projective $\epsilon$-lc pair of dimension $n$, 
	 \item $A$ is a very ample divisor on $F$ with $A^n\le r$,

	 \item $A-B$ is ample, and 

	 \item $M\ge 0$ is an $\rr$-Cartier $\rr$-divisor with $|A-M|_\rr\neq \emptyset$.
	\end{itemize}
	Then  	
	\[
	\lct(F,B,|M|_\rr)\ge \lct(F,B,|A|_\rr)\ge t.
	\]
\end{theorem}

Let us also recall the following immediate consequence of Birkar's proof of the BAB Conjecture, see \cite{bab16a, bir16bab}.

\begin{lemma}
\label{v.ample.index.lemma}
Fix a positive integer $n$ and a positive real number $\epsilon$.
Then, there exist positive integers $m_{n, \epsilon}=m(n, \epsilon)$, $s_{n, \epsilon}=s(n, \epsilon)$ such that for any $n$-dimensional $\epsilon$-lc Fano variety $X$, $-m_{n, \epsilon} K_X$ is a very ample Cartier divisor on $X$ of degree at most $s_{n, \epsilon}$.
\end{lemma}

In particular, $s_{n, \epsilon}$ also provides a bound for the volume of $-m_{n, \epsilon} K_X$. 
We can always assume that $m_{n, \epsilon} \geq 2$.

\begin{proof}
Theorem~\ref{bab.thm} and~\cite[Lemma 2.24]{bab16a} imply that there exists $t=t(n, \epsilon)$ such that $-tK_X$ is Cartier.
As $-K_X$ is ample, then~\cite[Theorem 1.1, Lemma 1.2]{K93} imply that there exists $l=l(n, \epsilon)$ such that $-ltK_X$ is very ample.
Hence, we can define $m_{n, \epsilon}:=ln$.
\\
The existence of $s_{n, \epsilon}$ follows from the previous part and the bound on the volume of $-K_X$ that was established in~\cite[Theorem~2.11]{bir16bab}.
\end{proof}

We are now ready to state the main result of this section.
We remind the reader that the definition of generalised pair can be found in~\S~\ref{sect.gen.pair}

\begin{theorem}
\label{log.birat.bnd.gen.pairs.thm}
Fix a positive real number $\epsilon$, a DCC set $I$,  and positive integers $n, p, d, l$. 
\\	
Let $\mathfrak{D}$ be the set of pairs $(X', \Delta'+\frac 1 l G)$ satisfying the following properties
\begin{enumerate}
	    \item there exists an $n$-dimensional projective generalized klt pair $(X \to X', \Delta', M)$,
	    \item the coefficients of $\Delta'$ belong to $I$, $pM$ is Cartier,
	    \item $K_{X'}+\Delta'+M' \equiv 0$, 
	    \item there exists $\Gamma \sim_\rr M$ such that 
	    $(X', \Delta'+\Gamma')$ is $\epsilon$-klt,
		\item there exists a contraction $f\colon X' \to Z$ on $X'$, with $\dim Z < \dim X'$, 
		\item $-K_{X'}$ is $f$-ample,
		\item there exists a very ample Cartier divisor $H$ such that $\Vol(Z,H)\leq d$, and 
	    \item $G \sim l f^\ast H$ and the coefficients of $\frac 1 l G$ are in $[0, 1]$.
	\end{enumerate}
Then $\mathfrak{D}$ forms a log birationally bounded family.
\end{theorem}

We postpone the proof of Theorem~\ref{log.birat.bnd.gen.pairs.thm} till the end of this subsection.

Theorem~\ref{log.birat.bnd.gen.pairs.thm} can be deduced from~\cite[Theorem~1.3]{bir.zhang} together with the following technical result which shows for suitable generalized pairs of Calabi--Yau type, endowed with a fibration, that the volume of the anticanonical class is bounded from above, thus, generalizing the result of Theorem~\ref{bound}.

\begin{proposition}
\label{hyperplane.generalized}
Fix a DCC set $I$, a positive real number $\epsilon$ and positive integers $n, p$. 
\\
If $(X \to X', \Delta', M)$ is a projective $n$-dimensional generalized klt pair such that
	\begin{enumerate}
\item the coefficients of $\Delta'$ are in $I$ and $pM$ is Cartier,
\item there exists a morphism $f\colon X' \to Z$, $0< \dim Z < \dim X'$, and $-K_{X'}$ is $f$-ample,
\item 
there exists a very ample Cartier divisor $H$ on $Z$ such that 
\[
K_{X'}+\Delta'+M' \sim_\mathbb{Q} f^\ast H, \ \textrm{and } \ K_Z+H \textrm{ is big},
\]
\item there exists $0 \leq \Gamma \sim_\rr M$ such that $(F, (\Delta'+\Gamma')\vert_F)$ is $\epsilon$-klt where $F$ is a general fibre of $f$ and $\Gamma'$ is the pushforward of $\Gamma$ on $X'$,

\end{enumerate}		
Then there exists a constant $k=k(I, \epsilon, p, n)$ such that
\[
\Vol(X', \Delta'+M') \leq k \Vol(Z, H).
\]
\end{proposition}

\begin{proof}
If $\Vol(X', \Delta'+M')=0$ there is nothing to prove.
Hence, we can assume that $\Delta'+M'$ is big and consequently that also its restriction to a general fibre of $f$ is big.
We can also assume that $X'$ is $\mathbb{Q}$-factorial by passing to a small $\mathbb{Q}$-factorial modification as that does not alter any of the hypotheses in the statement of the proposition.
\\
We will prove that there exists a constant $k'=k'(I, \epsilon, p, n, j)$ such that the conclusion of the proposition holds for $k'$ instead of $k$, where $j=\dim Z$.
The statement of the theorem then follows by taking 
\[
k(I, \epsilon, p, n):= \max_{0<j<n} k'(I, \epsilon, p, n, j).
\]
To define $k'$, we will work by induction on the set of ordered pairs of integers $(n, j), \; n>j>0$ with the lexicographic order.
\\
The pair  $(F, \Delta_F :=(\Delta' + \Gamma')|_F)$ is generalized $\epsilon$-klt and $K_F +  \Delta_F \sim _\rr 0$, by adjunction.
In particular, $F$ belongs to a bounded family, by Theorem~\ref{bab.thm}.
Let us also fix $\mfe=\mje=m(n-j, \epsilon) , \; \sfe=\sje=s(n-j, \epsilon)$ the two natural numbers defined in Lemma~\ref{v.ample.index.lemma}.
Apply Theorem~\ref{bnd.lct.thm} to the pair $(F, \Delta_F)$ with $A:=-m_{\dim F, \epsilon}K_F$, $r:=s_{\dim F, \epsilon}$, then we define $t=t(n-j, \sje, \epsilon)$.
In particular, for any effective divisor $G \in |-K_F|_{\mathbb{R}}$, $(F, \Delta_F+tG)$ is klt.
We define $\bar{t}:= \min\{\frac 1 2, \frac t 2\}$.

{\bf Claim}. ${\rm Vol}(X',\bar{t}(\Delta'+\Gamma')-2f^*H)=0$. 
\begin{proof}
In fact, if that is not the case, then $\bar{t}(\Delta'+\Gamma')-2f^*H$ is big.
Thus, there exists an effective $\mathbb{R}$-divisor $E\sim_{\mathbb{R}} \bar{t}(\Delta'+\Gamma')-2f^*H$. 
Let us define 
\[
\Theta:=(1-\bar{t})(\Delta'+\Gamma')+E,
\]
\noindent
so that $K_{X'}+\Theta\sim_{\mathbb{R}} -f^*H$.
Since $E\vert_F \sim_{\mathbb{R}} -\bar{t}K_F$ by construction, by the above discussion $(F, \Theta\vert_F)$ is klt; 
hence, $(X',\Theta)$ is klt over the generic point of $Z$, and, by Theorem~\ref{bundle}, there exists an effective divisor $\Gamma_Z$ on $Z$ such that 
\[
- f^*H \sim_{\mathbb{R}} K_{X'}+\Theta \sim_{\mathbb{R}} f^*(K_Z+\Gamma_Z).
\]
\noindent 
This immediately gives a contradiction since $K_Z+H$ is supposed to be big, while we have just proved that $K_Z+\Gamma_Z+H \sim_{\mathbb{R}} 0$.
\end{proof}
\noindent
We fix a general element $H' \in |H|$ and define $H'':=f^{-1}H'$, so that $H'' \sim f^\ast H$.
\\
Lemma \ref{jiang} implies that 
\[
{\rm Vol}(X',\bar{t}(\Delta'+\Gamma')) \leq 
{\rm Vol}(X',\bar{t}(\Delta'+\Gamma')-2f^\ast H)+2n{\rm Vol}(H'',\bar{t}(\Delta'+\Gamma')|_{H''}).
\]
Hence, as ${\rm Vol}(X',\bar{t}(\Delta'+\Gamma')-2f^\ast H)=0$ and $\Gamma' \sim M'$
\[
{\rm Vol}(X',\Delta'+M')\leq 2n \bar{t}^{-1} {\rm Vol}(H'',(\Delta'+M')|_{H''}).
\]
\noindent
If $j=1$, then $H''=\sum_{i=1}^hF_i$ where $h=\deg_Z(H)={\rm Vol}(Z, H)$ and $F_i$ is a general fibre of $f$ for each $i$.
Hence, $(F_i, (\Delta'+\Gamma')|_{F_i})$ is $\epsilon$-klt log Calabi--Yau pair of dimension $\dim X-1$. 
By Theorem \ref{bab.thm},
\[
{\rm Vol}(H'',(\Delta'+M')|_{H''})=\sum_{i=1}^h{\rm Vol}(F_i, -K_{F_i})\leq \frac{s_{n-1, \epsilon}}{m_{n-1, \epsilon}^{n-1}} {\rm Vol}(Z, H).
\]
Hence, we may take 
\[
k'(I,\epsilon, p, n, 1):=2n \bar{t}^{-1} \frac{s_{n-1, \epsilon}}{m_{n-1, \epsilon}^{n-1}},
\]
thus proving the starting step of the induction, as $j=1$.
\newline
When $j >1$, let us consider the map $g:=f|_{H''}$: $H''\rightarrow H'$. 
Then, by adjunction $H''$ supports a generalized klt pair $(H''' \to H'', \Delta'\vert_{H''}, M\vert_{H'''})$, such that 
\begin{itemize}
\item 
$H'''$ is a divisor in $X$, by possibly passing to a higher model of $X$,
\item 
$\dim H'=j-1 >0$
\item 
the coefficients of $\Delta'\vert_{H''}$ are in $I$ which is DCC, by adjunction and freeness of $|g^\ast H|$,
\item
$pM\vert_{H'''}$ is Cartier and $\Gamma\vert_{H'''}\sim_\mathbb{R} M\vert_{H'''}$, by restricting,
\item 
$(\Delta'+\Gamma')|_{H''}$ is $\epsilon$-klt along the general fiber of $g$, by adjunction and freeness of $|f^\ast H|$,
\item 
$K_{H''}= (K_{X'}+H'')\vert_{H''}$, by adjunction, since $H''$ is Cartier and general in a base point free linear system,
\item 
$-K_{H''}$ is $g$-ample since $H'' \sim_{g} 0$, and 
\item 
$K_{H''}+\Delta|_{H''}\sim_{\mathbb{R}}2g^*(H|_{H'})$, and $K_{H'} \sim (K_Z+H)|_{H'}$ is big.
\end{itemize}
Thus, by inductive hypothesis, since $\dim H'' =\dim X -1$,
\begin{eqnarray*}
{\rm Vol}(H'',(\Delta'+M')|_{H''}) & \leq k'(I,\epsilon, p, n-1, j-1){\rm Vol}(H',2H|_{H'})=\\
&=2^{j-1}k'(I, \epsilon, p, n-1, j-1){\rm Vol}(Z,H).
\end{eqnarray*}
\noindent
Taking \[
k'(I,\epsilon, p, n,j):=2^{j}n \bar{t}^{-1} k'(I,\epsilon, p, n-1, j-1),
\] 
the proof is complete.
\end{proof}

\begin{remark}
  From the point of view of the proof of Proposition \ref{hyperplane.generalized} there seems to be quite a few points where the strategy could be improved if we had a better knowledge and understanding of the features of generalized pairs.
\\
For example, it would be nice to show that it is possible to write down a canonical bundle formula for generalized pairs that would make it possible to drop hypothesis $(4)$ on the existence of an effective member of the linear system $|M|$ with bounded singularities.
 \\
For instance, for the purpose of the above proof, it would be enough to know that for a generalized klt pair $(X \to X', \Delta', M)$ and a morphisms $f \colon X' \to Z$ such that $K_{X'}+\Delta'+M' \sim_{f, \rr} 0$, there exists a klt pair $(Z, \Gamma)$ and a pseudo-effective divisor $P$ for which
\[
K_{X'}+\Delta'+M' \sim_{\rr} f^\ast(K_Z+ \Gamma+P).
\]
After the completion of this paper, it was shown in \cite{Fil18a, HL} that a canonical bundle formula exists for generalized pairs in full generality.
\end{remark}

In the course of the proof of Proposition~\ref{hyperplane.generalized} we used the following easy generalization of~\cite[Lemma 2.5]{Jia15}.

\begin{lemma}\label{jiang}
Let $X$ be a projective normal variety and let $D$ be a  $\mathbb{R}$-Cartier $\mathbb{R}$-divisor on $X$. 
Let $S$ be a base point free Cartier prime divisor on $X$.
Then for any rational number $q>0$
\[
\Vol(X, D+qS) \leq 
\Vol(X, D) + q\cdot \dim X \cdot \Vol(S, D|_S+qS|_S).
\]
\end{lemma}

\begin{proof}
	When $D$ is a $\mathbb{Q}$-Cartier $\mathbb{Q}$-divisor on $X$, 
	then this is the original statement of~\cite[Lemma 2.5]{Jia15}. 
\\
	When $D$ instead is an $\mathbb{R}$-Cartier $\mathbb{R}$-divisor, 
	then it suffices to choose a sequence of $\mathbb{Q}$-Cartier 
	$\mathbb{Q}$-divisors approximating $D$ in the N\'eron-Severi 
	group and notice that by virtue of continuity of the volume 
	function the inequality in the statement of the Lemma is 
	preserved in the limit.
\end{proof}

When working with log pairs, rather than generalized ones -- i.e., when $M = 0$ -- the statement of Proposition \ref{hyperplane.generalized} can be immediately reduced to the following one.

\begin{proposition}
\label{hyperplane}
Fix a DCC set $I \subset [0, 1)$ and an integer $n$. 
If $(X, \Delta)$ is a pair such that
	\begin{enumerate}
		\item $X$ is a projective variety of dimension $n$,
		\item $(X, \Delta)$ is klt and the coefficients of $\Delta$ are in $I$,
		\item there exists a morphism $f\colon X \to Z$, $0< \dim Z< \dim X $, and $-K_X$ is $f$-ample, and
		\item there exists a very ample Cartier divisor $H$ on $Z$ such that $K_X+\Delta \sim_\mathbb{R} f^\ast H$ and $K_Z+H$ is big,
	\end{enumerate}		
then there exists a constant $k=k(I, n)$ such that
	\[
		\Vol(X, \Delta) \leq k \Vol(Z, H).
	\]
\end{proposition}

\begin{proof}
Hypotheses (1)-(3) of Proposition~\ref{hyperplane.generalized} are automatically satisfied.
Corollary~\ref{elc} then implies that there exists $0<\epsilon=\epsilon(\dim X -\dim Y, I)$ such that $(F, \Delta\vert_F)$ is $\epsilon$-klt and this implies that condition (4) in Proposition~\ref{hyperplane.generalized} is satisfied as well, by adjunction to a general fibre of $f$.
\end{proof}

\begin{proof}[Proof of Theorem \ref{log.birat.bnd.gen.pairs.thm}]
Given a pair $(X',\Delta' + \frac 1 l G) \in \mathfrak{D}$, let us consider the generalized pair
\[
(X \to X', \Delta', M + \overline{H}),
\]
where $\overline{H}$ is the pullback of $H$ to $X$.
Since $(X \to X', \Delta', M)$ is generalized klt, the same holds for $(X \to X', \Delta', M + \overline{H})$. 
Moreover, we can assume that $X$ is a log resolution of $(X', \Delta' + \Gamma)$ where $M$ descends.
We will denote by $\widetilde G$ (resp. $\widetilde{\Delta'}$) the strict transform of $G$ (resp. $\Delta'$) on $X$, and by $\overline{G}:= \supp(\widetilde G)$.
In particular, since $G \sim l f^\ast H$ and the coefficients of $\frac 1 l G$ are in $[0, 1]$,
\begin{align}
\label{g.ineq}
\overline{G} \leq l\overline{H}.
\end{align}
\noindent
Let us fix  $H_1 \in |H|$ sufficiently general.
Let us define $H':=f^{-1} H_1$ on $X'$ and let $H''$ be the strict transform of $H'$ on $X$.
Thus, $H'' \in |\overline{H}|$, for a sufficiently general element. 
\\
By adjunction, up to choosing a sufficiently high model $X \to X'$, there exists a generalized pair structure $(H'' \to H',\Delta'\vert_{H'}, (M+\overline{H})\vert_{H''})$ satisfying the following conditions:
\begin{itemize}
\item[(i)] the coefficients of $\Delta'\vert_{H'}$ are in the DCC set $I$, since $H'$ is a general element of a base point free linear system,
\item[(ii)] $p(M +l\overline{H})\vert_{H''}$ is Cartier, $\forall l \in \mathbb{N}$
\item[(iii)] $\Gamma\vert_{H''} \sim_{\mathbb{R}} M\vert_{H''}$,
\item[(iv)] $(H', (\Delta'+\Gamma')\vert_{H'})$ is $\epsilon$-klt by  Bertini's theorem for log pairs.
\item[(v)] the restriction of $f$ to $H'$ gives a morphism $f_{H'}\colon H' \to H_1$, $\dim H_1 < \dim H'$,
\item[(vi)] the trace of $\overline{H}\vert
_{H''}$ on $H'$ is just $f_{H'}^\ast (H\vert_{H_1})$, by construction, and
\item[(vii)] $\forall l \in \mathbb{N}$ and  for the very ample Cartier divisor $H\vert_{H_1}$ on $H_1$, $K_{H'}+(\Delta'+\Gamma'+ l\overline{H})\vert_{H'} \sim_\mathbb{Q} (1+l)f^\ast_{H'} (H\vert_{H_1})$.
\end{itemize}
As $f$ is a $K_{X'}$-Mori fibre space, then for $0< \lambda \ll1$, $K_{X'}+(1+\lambda)(\Delta'+M')+f^\ast H$ is ample.
Hence, the divisor
\[
 K_X+E +\lceil \widetilde{\Delta'} \rceil+ 2M + \overline{H}
 \]
 is big, where $E$ is the exceptional divisor of the birational morphism $X \to X'$.
By~\cite[Theorem 1.3]{bir.zhang}, there exists a positive integer $m=m(I, n, p)$ such that for any $m'$ divisible by $m$ the linear system $|m'(K_X+E +\lceil \widetilde{\Delta'} \rceil+ 2M + \overline{H})|$ defines a birational map.
Up to passing to a higher birational model of $X$, we can assume that on $X$ 
\begin{align}
\label{eqn:N}
|m(K_X+E +\lceil \widetilde{\Delta'} \rceil+ 2M + \overline{H})|= |N| + F,
\end{align}
with $|N|$ is a base point free linear system inducing a birational morphism and $F$ the effective fixed part of the linear system.
\\
As the coefficients of $\Delta'$ vary in the DCC set $I$, there exists a positive integer $c=c(I)$ such that the coefficients of the divisor $c\Delta'$ are $\geq 1$ for any $(X', \Delta')$ in $\mathfrak{D}$. 
Thus,
\begin{align}
\label{eqn:c}
\Vol(X', c\Delta'+ M'+ f^\ast H) & = c^n\Vol(X', \Delta'+ \frac{M'+ f^\ast H}{c}) \\
\nonumber
&\leq c^n\Vol(X', \Delta'+ M' + f^\ast H).
\end{align}
\noindent
Lemma \ref{jiang} implies that 
\begin{align}
\nonumber
&\Vol(X', \Delta'+ M' + f^\ast H) \leq \\
\label{vol.chain.ineq.gen}
& \Vol(X', \Delta'+M')+  n \cdot \Vol(H', (\Delta' +M')|_{H'}+f^* H|_{H'})=\\
\nonumber
& \Vol(X', \Delta'+M')+  n \cdot \Vol(H', (\Delta' +M')|_{H'}+f^*_{H'} (H|_{H_1})).
\end{align}
Hence, 
  \begin{align} 
\nonumber
&\Vol(X, K_X+E +\lceil \widetilde{\Delta'} \rceil+ 2M + \overline{H}) \\
\nonumber
\leq  & 
\Vol(X', K_{X'}+\lceil \widetilde{\Delta'} \rceil+ 2M' + H')
& \hfill[\text{using $f_{\ast}$}]
\\ 
\label{vol.ineq}
\leq  & \Vol(X', c\Delta'+ M'+H') 
& \hfill[K_{X'} + \Delta'+M' \equiv 0]  
\\ 
\nonumber 
\leq & c^n \Vol(X', \Delta'+M'+H')
& \hfill[\textrm{by } \eqref{eqn:c}]
\\
\nonumber 
\leq &  c^{n} [\Vol(X', \Delta'+M') +
\\
\nonumber
& n \cdot \Vol(H', (\Delta'+M')|_{H'}+ f^*_{H'} (H|_{H_1})],
	\end{align}
where the last inequality follows from~\eqref{vol.chain.ineq.gen}.\\

{\bf Claim 1}. 
{\it There exists a positive real number $C_1=C_1(n, \epsilon)$ such that }
\[\Vol(X', \Delta'+M') \leq C_1.\]
\begin{proof}
We can assume that $\Vol(X', \Delta'+M') >0$, otherwise there is nothing to prove.
We can also assume that $X'$ is $\mathbb{Q}$-factorial, as passing to a small $\mathbb{Q}$-factorialization does not change volumes.\\
Hence, for $0<\eta\ll 1$, $(X', (1+\eta)(\Delta' + \Gamma'))$ is klt and $K_{X'}+(1+\eta)(\Delta' + \Gamma') \sim_\mathbb{R} \eta(\Delta' + \Gamma')$ is big.
Thus, we can run the $(\eta(\Delta' + \Gamma'))$-MMP with scaling of an ample divisor which terminates with a birational model $\psi \colon X' \dashrightarrow X''$ such that $\Delta'' + \Gamma''$ is ample on $X''$ and $\Vol(X'', \Delta'' + \Gamma'') = \Vol(X', \Delta' + \Gamma')$, where $\Delta'' := \psi_\ast \Delta'$, $\Gamma'':=\psi_\ast \Gamma'$.
On the other hand,  as $K_{X'}+\Delta' + \Gamma' \sim_\mathbb{R} 0$ and $(X', \Delta' + \Gamma')$ is $\epsilon$-klt, then $X''$ is an $\epsilon$-klt Fano variety and $\Vol(X'', \Delta'' + \Gamma'') = \Vol(X'', -K_{X''})$.
Hence, it suffices to take $C_1$ to be the constant whose existence is implied by~\cite[Theorem 2.11]{bir16bab}.
\end{proof}

{\bf Claim 2}. {\it There exists a positive real number $C_2=C_2(I, \epsilon, p, n)$ such that }
\[\Vol(H', (\Delta'+M')\vert_{H'} + f^\ast H\vert_{H'})  \leq C_2.\]
\begin{proof}
We can assume that $\Vol(H', (\Delta'+M')\vert_{H'} + f^\ast H\vert_{H'}) >0$, otherwise there is nothing to prove.
We can assume that $H'$ is $\mathbb{Q}$-factorial, as passing to a small $\mathbb{Q}$-factorialization does not change volumes.
We distinguish 2 cases.\\
If $\dim H_1=0$, then $H'$ is the disjoint union of $s$ fibres of $f_{H'}$, $F_1, \dots, F_s$, where $s:=\Vol(Z, H)\leq d$. 
Furthermore, in this case $f^\ast_{H'} (H|_{H_1}) = 0$ so that  
\begin{align*}
& \Vol(H', (\Delta'+M')\vert_{H'} + f^\ast_{H'} (H|_{H_1})) =\Vol(H', (\Delta'+M')\vert_{H'})=\\ 
& \sum_{i=1}^s(\Delta'+M')^{\dim X'-\dim Z}  \cdot F_i = s (\Delta'+M')^{\dim F}  \cdot F,
\end{align*}
where $F$ is a general fibre of $f\vert_{H'}$ (equivalently, of $f$).
By hypotheses (3)-(6) in the statement of the theorem, $F$ is an $\epsilon$-klt Fano and $K_{F} \sim_\mathbb{R} (\Delta'+M')\vert_{F}$; hence, by~\cite[Theorem 2.11]{bir16bab} there exists a constant $\bar{C}=\bar{C}(\epsilon, \dim F)$ such that 
\begin{align*}
(\Delta'+M')\vert_{H'}^{\dim F}  \cdot F = \Vol(F, (\Delta'+M')\vert_{F}) = \Vol(F, -K_F) \leq \bar{C}.
\end{align*}
We define $C_{2, a}:=\max_{0<i<\dim X'} \bar{C}(\epsilon, i)$.
\\
If $\dim H_1> 0$, by properties (vi-vii) in the previous page
\begin{align*}
K_{H'}+(\Delta'+M')\vert_{H'} + l f_{H'}^\ast (H\vert_{H_1}) \sim_\mathbb{R}  f_{H'}^\ast ((1+ l)H\vert_{H_1}), \ \forall l \geq 0.
\end{align*}
Moreover, Kawamata-Viehweg vanishing and the fact that $H\vert_{H_1}$ is very ample imply that $K_{H_1}+(n+1)H\vert_{H_1}$ is big: in fact, for some $m \in \{1, \dots, n\}$, $h^0(H_1, K_{H_1}+mH\vert_{H_1}) \neq 0$ as $\chi(H_1, K_{H_1}+mH\vert_{H_1})= h^0(H_1, K_{H_1}+mH\vert_{H_1})$ by Kawamata-Viehweg vanishing.
Hence, by properties (i-vii) in the previous page, using Proposition \ref{hyperplane.generalized}, we conclude that there exists $k=k(I, \epsilon, p, n-1)$ such that 
\begin{align*}
&\Vol(H', (\Delta'+M')\vert_{H'}+ f_{H_1}^\ast (H\vert_{H_1}) )
\leq \Vol(H', (\Delta'+M')\vert_{H'}+ n f_{H'}^\ast (H\vert_{H'}) )
\\
&\leq k (n+1)^{n-1} \Vol(H_1, H\vert_{H_1}) \leq k (n+1)^{n-1} d,
\end{align*}
\noindent
as $\Vol(H_1, H\vert_{H_1}) = \Vol(Z, H) \leq d$ since $H_1 \in |H|$.
We define $C_{2, b}:=k (n+1)^{n-1} d$.
\\
To terminate the proof of the claim, it suffices to take $C_2:= \max\{C_{2, a}, C_{2, b}\}$.
\end{proof}
\noindent
By~\cite[Lemma 2.4.2(3)]{HMX13} to show that $\mathfrak{D}$ forms log birationally bounded, it suffices to show that the set of pairs $(X, \lceil \widetilde{\Delta'} + \widetilde{G} \rceil+ E)$ is log birationally bounded.
Since $|N|$ in~\eqref{eqn:N} is a base point free linear system that induces a birational morphism, by~\cite[Lemma 2.4.2(4)]{HMX13} it suffices to show that there exists positive real numbers $C_3=C_3(I, \epsilon, p, n)$, $C_4=C_4(I, \epsilon, p, n)$ such that 
\[
\Vol(X, N) \leq C_3 \textrm{ and } (\lceil \widetilde{\Delta'}\rceil+ E+\overline G) \cdot N^{n-1} \leq C_4.
\]
Claims 1-2 and \eqref{vol.ineq} imply that 
\begin{align}
\label{claim.ineq}
\Vol(X, K_X+E +\lceil \widetilde{\Delta'} \rceil+ 2M + \overline{H}) \leq c^n(C_1+C_2).
\end{align}
\noindent
By construction and~\eqref{claim.ineq}
\[
\Vol(X, N) \leq m^n\Vol(X, K_X+E +\lceil \widetilde{\Delta'} \rceil+ 2M + \overline{H}) \leq (cm)^n (C_1+C_2).
\]
\noindent
Thus, it suffices to take $C_3:=(cm)^n (C_1+C_2)$.\\
By \cite[Lemma 3.2]{HMX13}, denoting $N'=2(2n+1)N$
\begin{align}
\label{hmx13.ineq1}
&(\lceil \widetilde{\Delta'}\rceil+ E + \overline G) \cdot N^{n-1} \leq 
(\lceil \widetilde{\Delta'}\rceil+ E+\overline G) \cdot (N')^{n-1}\\
\nonumber 
 \leq &  2^n \Vol(X, K_X+\lceil \widetilde{\Delta'}\rceil+ E+\overline G+ N'). 
\end{align}
Since $|m(K_X+\lceil \widetilde{\Delta'}\rceil+ E+ 2M + \overline{H}))|  = |N| + F$ and $F \geq 0$, then using~\eqref{g.ineq}
\begin{align}
\label{hmx13.ineq2}
&\Vol(X, K_X+\lceil \widetilde{\Delta'}\rceil+ E+ \overline G+ N') 
\\
\nonumber
\leq &\Vol(X, K_X+\lceil \widetilde{\Delta'}\rceil+ E+ \overline G+ 2(2n+1) (K_X+\lceil \widetilde{\Delta'}\rceil+ E+ 2M + \overline{H}))\\
\nonumber
\leq &\Vol(X, K_X+\lceil \widetilde{\Delta'}\rceil+ E+ l\overline{H}+ 2(2n+1) (K_X+\lceil \widetilde{\Delta'}\rceil+ E+ 2M + \overline{H})).
\end{align}
Thus,~\eqref{hmx13.ineq1}-\eqref{hmx13.ineq2} imply that
\begin{align*}
&(\lceil \widetilde{\Delta'}\rceil+ E + \overline G) \cdot N^{n-1}
\\
\leq & 2^n \Vol(X, K_X+\lceil \widetilde{\Delta'}\rceil+ E+l\overline{H} + 2(2n+1) m(K_X+\lceil \widetilde{\Delta'}\rceil+ E+ 2M + \overline{H}))  \\
\leq & 2^n \Vol(X, (2(2n+1)m+l) (K_X+ \lceil\widetilde{\Delta'}\rceil+ E+ 2M + \overline{H})),
\end{align*}
where the last inequality follows from the fact that $M$ and $K_X+ \lceil\widetilde{\Delta'}\rceil+ E+ 2M$ are pseudoeffective.
By~\eqref{claim.ineq} it suffices to take $C_4:= 2^n(2(2n+1)m+l)^nc^n (C_1+C_2)$.
\end{proof}

\subsection{Boundedness for Mori fibre spaces of log Calabi--Yau type}

We can now use the results from the previous subsection to show that, in the case of klt Calabi--Yau pairs endowed with a Mori fibre space structure with bounded base and singularities, we can produce a very ample divisor with bounded volume on the total space using the one on the base. 

After the completion of this paper, a generalization of this result appeared in~\cite[Theorem~4.6]{rccy3} and later that was further improved in~\cite[Theorem~1.2]{birkar.lcy.fibr}.

\begin{theorem}\label{bound.MFS.thm}
	Fix a DCC set $I\subset [0,1]$ and positive integers $n, d$. 
	Then the set $\mathfrak{D}$ of log pairs $(X, \Delta)$
	satisfying 
	\begin{enumerate}
		\item $(X, \Delta )$ is a projective klt pair of dimension $n$,
		\item the coefficients of $\Delta$ belong to $I$,
		\item $K_X+\Delta \equiv 0$,
	  	\item there exists a contraction $f\colon X\to Z$ with $\dim Z < \dim X$, 
	  	\item $-K_X$ is $f$-ample, and
		\item there exists a very ample Cartier divisor $H$ on $Z$ with $\Vol(Z,H)\leq d$,
	\end{enumerate}
	forms a bounded family.
\end{theorem}	  

\begin{proof}
By Theorem~\ref{hmx_1.5.thm}, we can assume that $I$ is finite.
By Corollary~\ref{elc}, there exists $\epsilon(n, I) >0$ such that $(X, \Delta)$ is $\epsilon$-lc.
As $f^* H$ is base point free,  by Bertini's theorem for log pairs, cf.~\cite[Theorem~4.8]{Kol97}, there exists an effective prime divisor $0 \leq G\sim 2f^* H$ such that $(X,\Delta+\frac 1 2 G)$ is $\epsilon'$-lc, where $\epsilon'=\min\{\frac 1 2, \epsilon\}$; 
furthermore, the coefficients of $\Delta+\frac 1 2 G$ are contained in $I'=I \cup \{\frac 1 2\}$ which is still a DCC set.
For any pair $(X, \Delta) \in \mathfrak{D}$, we fix such a choice of $G=G(X, \Delta)$.
\\
We define the set of pairs $\mathfrak{D}'$ in the following way
\begin{align*}
\mathfrak{D'}:=\{ & (X, \Delta+\frac 1 2 G) \ \vert \ (X, \Delta) \in \mathfrak{D}, \ G \textrm{ is the divisor defined above}\}.
\end{align*}
The set $\mathfrak{D}'$ is log birationally bounded: 
in fact, pairs $(X, \Delta+\frac 1 2 G) \in \mathfrak{D}'$ satisfy the hypotheses of Theorem~\ref{log.birat.bnd.gen.pairs.thm}, since 
\begin{itemize}
\item 
$(X,\Delta)$ is trivially a generalized pair by imposing $M=0$,
\item 
it suffices to fix the integer $l$ in the statement of Theorem~\ref{log.birat.bnd.gen.pairs.thm} to be $2$,
\item 
$I$ is a DCC set, 
\item 
$-K_X$ is $f$-ample, and 
\item 
$(X, \Delta)$ is $\epsilon$-lc.
\end{itemize}	
\noindent
Since $-K_X$ is $f$-ample, for any $0<\eta \ll 1$,
\[
K_X + (1+\eta) \Delta+\frac 1 2 G \sim_\mathbb{R} \eta \Delta + \frac 1 2 G
\] 
is an ample divisor and the pair $(X, (1+\eta) \Delta+\frac 1 2 G)$ is $\frac{\epsilon'}{2}$-klt.
For any pair $(X, \Delta+ \frac 1 2 G) \in \mathfrak{D}'$, we fix such a choice of $\eta=\eta(X, \Delta, G)$.
\\
We define the set of pairs $\mathfrak{D}''$ in the following way
\begin{align*}
\mathfrak{D}'':=\{ & (X, (1+\eta)\Delta+\frac 1 2 G) \ \vert \ (X, \Delta+\frac 1 2 G) \in \mathfrak{D}', \\ 
& \eta \textrm{ is the positive real number defined above}\}.
\end{align*}
As $\mathfrak{D}'$ is log birationally bounded,  and the notion of log birational boundedness does not depend on the coefficients of the divisor of the pairs in $\mathfrak{D}'$, see Definition~\ref{def.bound}, then $\mathfrak{D}''$ is also log birationally bounded.\\

{\bf Claim}.
The set $\mathfrak{D}''$ satisfies the hypotheses of Theorem \ref{hmx_1.6}. 
\begin{proof}
As the coefficients of $\Delta+\frac 1 2 G$ are contained in a DCC set, there exists $\delta=\delta(I) >0$ such that the coefficients of $\Delta+\frac 1 2 G$ are $\geq \delta$; 
a fortiori, the same must hold for the coefficients of $(1+\eta)\Delta+\frac 1 2 G$.
By construction, $K_X+(1+\eta)\Delta+\frac 1 2 G$ is ample and by construction $(X, (1+\eta)\Delta+\frac 1 2 G)$ is $\frac{\epsilon'}{2}$-klt, for $\epsilon'=\epsilon'(n, I)$.
\end{proof}
Applying Theorem~\ref{hmx_1.6} to $\mathfrak{D}''$, then $\mathfrak{D}''$ is a bounded set of pairs. 
In particular $X$ belongs to a bounded family:
by the definition of boundedness of $\mathfrak{D}''$, see Definition~\ref{def.bound}, in particular, there exists a projective morphism $f \colon Z \to T$ such that $T$ is of finite type and any variety $X \in \mathfrak{D}''$ appears also as one of the fibers of $f$.
Hence, up to decomposing $T$ as a finite disjoint union of locally closed subsets, we can assume that $Z \subset \mathbb{P}^n \times T$, which implies, by generic flatness, that there exists a positive integer $d=d(\mathfrak{D}'')$ such that for any $X \in \mathfrak{D}''$ one can find a very ample Cartier divisor $H_X$ on $X$ satisfying $H_X^{\dim X} \leq d$.
Since the coefficients of $\Delta$ belong to the finite set $I$ and $\Delta \equiv -K_X$, then the set $\mathfrak{D}$ is a bounded set, again, by the existence of Chow varieties,~\cite[Theorem~3.12]{Kol96}, since the degree with respect to $H_X$ of each component of $\Delta$ is bounded.
\end{proof}


\section{Proof of the Theorems}
\label{sec.proofs}

In this section we prove the theorems stated in the Introduction.

\begin{proof}[Proof of \ref{notproduct}]
From Theorem \ref{structure}, after a birational contraction,  we obtain a tower of Mori fiber spaces
\begin{align*}
\xymatrix{
	X \ar@{-->}[r] &X'=Y_0 \ar[r]^{p_0} & Y_1 \ar[r]^{p_1} & Y_2 \ar[r]^{p_2} &\dots \ar[r]^{p_{k-1}}& Y_{k} \ar[r]^{p_{k}}& Y_{k+1}.
}
\end{align*}
Since $X$ is not of product type, then the same theorem implies that the variety $Y_{k+1}$ is a point and, thus, $Y_k$ must be Fano.
\end{proof}

\begin{proof}[Proof of \ref{bounded.bases}]
By Theorem~\ref{hmx_1.5.thm}, we can assume that the coefficients of $\Delta$ belong to a finite set $I_0 \subset I$.
\\
As the set of varieties $Z$ in the statement forms the bounded family $\mathfrak{F}$, 
there exists a projective morphism $f \colon \mathcal Z \to T$ such that $T$ is of finite type and any variety $Z \in \mathfrak F$ appears as one of the fibers of $f$.
Hence, up to decomposing $T$ as a finite disjoint union of locally closed subsets, we can assume that $\mathcal Z \subset \mathbb{P}^n \times T$, for some fixed $n=n(\mathfrak F)$; 
thus, by generic flatness, that there exists a positive integer $d=d(\mathfrak{F})$ such that for any $Z \in \mathfrak F$ there exists a very ample Cartier divisor $H_Z$ on $Z$ satisfying $H_Z^{\dim Z} \leq d$.
\newline
Then, the theorem follows from Theorems~\ref{bound.MFS.thm}.
\end{proof}

\begin{proof}[Proof of \ref{main}]
We prove the Theorem by induction on $n \leq 4$.
By Theorem~\ref{hmx_1.5.thm}, we can assume that the set $I$ is finite.\\
The case $n=1$ is immediate as $X$ can only be $\mathbb{P}^1$ and $I$ is finite.
\\
Let us assume that the Theorem holds for $n-1$.
Let $(X, \Delta)$ be a pair in $\mathfrak{D}$.
By Theorem \ref{structure}, $X$ is birational to a tower of Mori fibre spaces
\begin{equation}
\label{diag.pf.1_3}
\xymatrix{
X \ar@{-->}[r]^{r} &X'=Y_0 \ar[r]^{p_0} & Y_1 \ar[r]^{p_1} & Y_2 \ar[r]^{p_2} &\dots \ar[r]^{p_{k-1}}& Y_{k}.
}
\end{equation}
For $i>0$, if $\dim Y_i>0$, then by Theorem~\ref{bundle}, there exists an effective divisor $\Delta_i$ on $Y_i$ such that $(Y_i, \Delta_i)$ is a klt Calabi--Yau pair.
The divisor $\Delta_i \neq 0$, otherwise the lc-trivial contraction $p_{i-1} \circ \dots \circ p_1\circ p_0 \colon (X', \Delta') \to Y_i$ would be strongly generically isotrivial by Theorem~\ref{ambro}, cf. Remark~\ref{rmk.bir.isotr};
by Definition~\ref{defproduct}, the rational contraction $p_0 \circ r \colon X \to Y_1$ would contradict the assumption that $(X, \Delta)$ is not of product type.
\\
We denote by $\mathfrak{F}$ the set of pairs $(X',\Delta'=r_\ast \Delta)$ just constructed.
We first prove that the set $\mathfrak{F}$ forms a bounded family.
By Theorem~\ref{bounded.bases}, it suffices to show that the set $\mathfrak{F}_1$ given by varieties $Y_1$ appearing in~\eqref{diag.pf.1_3} forms a bounded family, since the morphism $p_0$ in~\eqref{diag.pf.1_3} is a $K_{X'}$-Mori fibre space.
To this end, we distinguish several cases depending on the dimension of $Y_1$.
\newline
If $Y_1$ is a point, then $Y_1$ is certainly bounded.
\\
If $\dim Y_1 =1$, then $Y_1 = \mathbb{P}^1$, since $\Delta_1 \neq 0$; thus, $Y_1$ is bounded.
\\
If $\dim Y_1 =2$, then by Theorem~\ref{birk.thm} there exists an effective divisor $\Delta_1 \neq 0$ on $Y_1$ and a positive real number $\epsilon=\epsilon(N, I)$ such that $(Y_1, \Delta_1)$ is an $\epsilon$-lc Calabi--Yau surface pair.
Then, boundedness of $Y_1$ then follows from Theorem~\ref{alexeev.thm}.
\\
If $\dim Y_1=3$, then by Theorem~\ref{semiample} and Lemma~\ref{acc.lemma}, we can assume that $\Delta_1$ is a divisor on $Y_1$ with coefficients in a DCC set $J=J(n, I)$.
Furthermore, by Theorem~\ref{hmx_1.5.thm}, we can assume that the set $J$ is finite.
If $Y_2$ is a point, then $Y_1$ is bounded by Corollary~\ref{elc} and Theorem~\ref{hmx_1.7.thm}.
If $\dim Y_2 =1$, then $Y_2 = \mathbb{P}^1$ as $\Delta_2 \neq 0$. Hence, $Y_2$ is bounded.
If $\dim Y_2=2$, by Theorem~\ref{birk.thm} we can assume that there exists a positive real number $\epsilon$ such that $(Y_2, \Delta_2)$ is an $\epsilon$-lc Calabi--Yau surface pair; as $\Delta_2 \neq 0$ boundedness of $Y_2$ follows from Theorem~\ref{alexeev.thm}.
In both cases $\dim Y_2=1, 2$, boundedness of $Y_1$ follows from Theorem~\ref{bounded.bases} as $Y_2$ belongs to a bounded family.
\\
As the set $\mathfrak{F}$ forms a bounded family,~\cite[Proposition 2.5]{HX14} implies that there exists a bounded family $(Z,D)\rightarrow T$ such that for any pair $(X',\Delta') \in \mathfrak{F}$ and for any set $\{E_1, \dots, E_k\}$ of exceptional divisors of discrepancy at most $0$ over $(X', \Delta')$, there exists $t\in T$ and a birational morphism $\mu_t : Z_t \rightarrow X'$ which extracts precisely $\{E_1, \dots, E_k\}$ and $K_{Z_t}+D_t=\mu_t^* (K_{X'} +\Delta')$. 
We define the set $\mathfrak{D}'$ to be the set of klt Calabi--Yau pairs $(Z_t, D_t)$ just constructed.
Given $(X, \Delta) \in \mathfrak{D}'$ and the corresponding pair $(X', \Delta') \in \mathfrak{F}$ constructed in~\eqref{diag.pf.1_3}, since $K_X +\Delta \sim_\rr 0$ and $K_{X'} +\Delta' \sim_\rr 0$, all the exceptional divisors $E_1, \dots, E_l$ for $r \colon X \dashrightarrow X'$ have discrepancy $\leq 0$ with respect to $(X', \Delta')$.
Hence, by the above construction, there exists $(Z_t, D_t) \in \mathfrak{D}'$ crepantly extracting the divisors $E_1, \dots, E_l$; 
hence, by construction, $(Z_t, D_t)$ is isomorphic in codimension one to $(X, \Delta)$, which proves the second part of the statement of the theorem.
As $\mathfrak{D}'$ forms a bounded family and pairs in $\mathfrak{D}$ are isomorphic in codimension one to pairs in $\mathfrak{D}'$, ~\cite[Lemma~2.4.2(3)]{HMX13} implies that $\mathfrak{D}$ is log birationally bounded. 
\end{proof}

\begin{proof}[Proof of \ref{elliptic}]
For the sake of presentation we divide this proof into steps.
\newline
For $2\leq n \leq 5$, we denote by $\mathfrak{E}_n$ the set of $n$-dimensional elliptic fibrations $f \colon Y \to X$ satisfying the hypotheses in the statement of the theorem.
We fix the value of $n$.
We denote by $S$ the Zariski closure of the rational section of $f$ whose existence is assumed in the statement of the theorem.\\

{\bf Step 0}. {\it In this step we show that there exists a finite set $I \subset [0, 1) \cap \mathbb{Q}$ which only depends on $n$ such that for any normal variety $X$ that is the base of an elliptic fibrations $f \colon Y \to X$ in $\mathfrak{E}_n$, there exists an effective divisor $\Delta$ on $X$ with coefficients in $I$ and $(X, \Delta)$ is a klt Calabi--Yau pair}.
\newline
Given an elliptic fibration $f \colon Y \to X \in \mathfrak{E}_n$, as $K_Y \sim_\mathbb{Q} 0$ and $Y$ is terminal, the canonical bundle formula, see Theorem \ref{bundle}, implies that 
\[
K_Y \sim_\mathbb{Q} f^\ast(K_X+B_X+M_X),
\]
where $B_X$ is the boundary part and $M_X$ is the moduli part.
As the generic fiber of $f$ has dimension and genus one, Theorem \ref{semiample} implies that there exists $m=m(n)$ for which the moduli part $|mM_{\hat{X}}|$ becomes base point free, on a suitable birational model $\hat{X} \to X$ of the base of the elliptic fibration.
By Lemma~\ref{acc.lemma}, the coefficients of $B_X$ in turn belong to a DCC set $I'=I'(n) \subset \mathbb{Q}$.
By choosing a sufficiently general member $\hat{M} \in |mM_{\hat{X}}|$, we can assume that the pair $(X, \Delta)$, $\Delta:= B_X+\frac{1}{m}M$ is a klt Calabi-Yau pair, where $M$ is the pushforward of $\hat{M}$ on $X$; 
moreover as the coefficients of $\Delta$ belong to the DCC set $I' \cup \{\frac 1 m\}$, Theorem \ref{hmx_1.5.thm} implies that those actually lie in a finite set $I \subset I' \cup \{\frac 1 m\}$.
As $X$ is not of product type, Lemma~\ref{lem:cy.not.prod.type} implies that $K_X \not \sim_\mathbb{Q} 0$, hence, $\Delta \neq 0$.
\\

\noindent
We denote by $\mathfrak{D}_n$ the set of $(n-1)$-dimensional klt Calabi--Yau pairs $(X, \Delta)$ that we have just constructed.\\

{\bf Step 1}. {\it In this step we prove that there exists a bounded set of pairs $\mathfrak{D}'_n$ such that any pair $(X, \Delta) \in \mathfrak{D}_n$ is isomorphic in codimension $1$ to a $\mathbb{Q}$-factorial klt Calabi--Yau pair $(X', \Delta') \in \mathfrak{D}'_n$.}
\newline
The hypothesis on the $\mathbb{Q}$-factoriality of the $X'$ will be used in Step 5.
\newline
Since for any Calabi--Yau pair $(X, \Delta) \in \mathfrak{D}_n$, $X$ is not of product type and the coefficients of $\Delta$ lie in a finite set $I$, by Step 0, Theorem \ref{main} implies that there exists a bounded set of pairs $\overline{\mathfrak{D}}_n$ such that any $(X, \Delta) \in \mathfrak{D}_n$ is isomorphic in codimension 1 to a klt Calabi--Yau pair $(\overline{X}, \overline{\Delta}) \in \overline{\mathfrak{D}}_n$.
Hence, by definition of boundedness there exists a positive integer $\overline{C} = \overline{C}(n, I)$ and a very ample Cartier divisor $\overline{H}$ on $\overline{X}$ such that $\overline{H}^{\dim \overline{X}} \leq \overline{C}$.
Given a pair $(X, \Delta) \in \mathfrak{D}_n$ we define a pair $(X', \Delta') \in \mathfrak{D}'_n$ as follows:
\begin{itemize}
\item[a)] if the variety $\overline{X}$ defined above is $\mathbb{Q}$-factorial, we take $(X', \Delta') := (\overline{X}, \overline{\Delta})$;
\item[b)] otherwise, we consider a small $\mathbb{Q}$-factorialization $\overline{\pi} \colon X' \to \overline{X}$ of $\overline{X}$ together with the strict transform $\Delta'$ of $\overline{\Delta}$ on $X'$.
\end{itemize}
By construction, we know that the varieties from a) belong to a bounded set.
To show boundedness for the pairs $(X', \Delta')$ in case b), let $H_1 \in |2\overline{\pi}^\ast\overline{H}|$ be a sufficiently general member of base point free linear system of $2\overline{\pi}^\ast\overline{H}$.
The pair $(X', \Delta' + \frac 1 2 H_1)$ is klt by Bertini's theorem and $K_{X'} +\Delta' + \frac 1 2 H_1$ is big and nef.
The volume of $K_{X'} +\Delta' + \frac 1 2 H_1$ coincides with the volume of $\overline{H}$, hence it is an integer that is at most $\overline{C}$.
Thus,~\cite[Theorem~6]{MST16} implies that the pair $(X', \Delta' + \frac 1 2 H_1)$ and thus also the pair $(X', \Delta')$ belong to a bounded family of pairs.
Hence, for any pair $(X', \Delta') \in \mathfrak{D}'_n$ there exists a positive integer $C' = C'(n, I)$ and a very ample Cartier divisor $H'$ on $X'$ such that $H'^{n-1} \leq C'$.\\

{\bf Step 2}. {\it In this step we show that for any elliptic fibration $f \colon Y \to X$ in $\mathfrak{E}_n$ there exists an elliptic fibration $f' \colon Y' \to X'$ satisfying the following properties:
\begin{enumerate}
\item there exists $\Delta'$ such that $(X', \Delta') \in \mathfrak{D}'_n$;
\item $Y'$ is projective, $\mathbb{Q}$-factorial, terminal, $K_{Y'} \sim_\mathbb{Q} 0$, and $Y'$ is isomorphic to $Y$ in codimension one.
\end{enumerate}
}
\noindent
Given $f \colon Y \to X$ in $\mathfrak{E}_n$, there exists a $\mathbb{Q}$-factorial klt Calabi--Yau pair $(X', \Delta')$ in $\mathfrak{D}'_n$ isomorphic in codimension $1$ to the pair $(X, \Delta)$ constructed in Step 1.
By Proposition~\ref{small.qfact.prop}, given a small $\mathbb{Q}$-factorialization $\widetilde{X} \to X$ of $X$, there exists a terminal $\mathbb{Q}$-factorial projective elliptic variety $\widetilde{f} \colon \widetilde{Y} \to \widetilde{X}$, and $\widetilde{Y}$ is isomorphic to $Y$ in codimension one.
As $\widetilde{X}$ and $X'$ are isomorphic in codimension $1$ and $\widetilde{X}$ is $\mathbb{Q}$-factorial by construction, the existence of a fibration $f' \colon Y' \to X'$ satisfying (2) is guaranteed by Proposition~\ref{bir.contr.cy.prop}.
In particular, $K_{Y'} \sim_{\mathbb{Q}} 0$.\\

\noindent
We denote by $\mathfrak{E}'_n$ the set of elliptic Calabi--Yau fibrations $f' \colon Y' \to X'$ that we constructed in Step 2.
Given any $f \colon Y \to X \in \mathfrak{E}_n$ and the corresponding $f' \colon Y' \to X' \in \mathfrak{E}'_n$, we will denote by $S'$ the strict transform of $S$ on $Y'$. 
The irreducible divisor $S'$ is still a rational section for the elliptic fibration $f' \colon Y' \to X'$ since $f'$ is isomorphic to $f$ over a dense open set of $X'$ by construction.\\

{\bf Step 3}. {\it In this step we show that for any elliptic fibration $f'\colon Y' \to X'$ in $\mathfrak{E}'_n$ we can run a $S'$-MMP over $X'$ with scaling of an ample divisor. 
We show that this MMP terminates with a relatively minimal model $f_m \colon Y^m \to X'$.
Moreover, we show that $Y^m$ is terminal, $\mathbb{Q}$-factorial and isomorphic to $Y'$ in codimension $1$.
}
\newline
Given an elliptic fibration $f'\colon Y' \to X'$ in $\mathfrak{E}'_n$, as $S'$ is a rational section for $f'$ and $Y'$ is $\mathbb{Q}$-factorial, $S'$ is relatively big over $X'$.
As $K_{Y'} \sim_\mathbb{Q} 0$, for any positive real $\gamma$, 
\begin{equation}
\label{1_1.eq1}
K_{Y'} + \gamma S' \sim_\mathbb{R} \gamma S'.
\end{equation}
Since for a general fibre $F$ of $f'$, $S'\cdot F=1$, then $K_{Y'} + \gamma S'$ is big over $X'$, and for any $0 < \gamma \ll 1$, $(Y', \gamma S')$ is klt.
Thus, we can run the $(K_{Y'} + \gamma S')$-MMP with scaling of an ample divisor over $X'$ and by~\eqref{1_1.eq1} this is also a run of the $S'$-MMP.
By~\cite[Corollary 1.4.2]{BCHM}, this run of the MMP terminates with a relatively minimal model $Y' \dashrightarrow Y^m$ since $S'$ is relatively big.
The strict transform $S_m$ of $S'$ on $Y^m$ is big and nef over $X'$.
As $Y^m$ is a relatively minimal model, it follows that it is $\mathbb{Q}$-factorial. 
\\
By \eqref{1_1.eq1}, any extremal ray that is contracted in the run of the MMP in Step 3 has negative intersection with the strict transform of $S'$.
As $S'$ is irreducible, then each step of the MMP must be a flip:
in fact, if that were not the case, then there would be some step of the MMP which is a divisorial contraction and the exceptional divisor would then be the strict transform of $S'$; 
this is not possible since $S'$ is big over $X'$ and it is irreducible.
Hence, $K_{Y^m}\sim
_\mathbb{Q} 0$, and $Y'$ and $Y^m$ are isomorphic in codimension one.
Since the relative dimension of $f' \colon Y' \to X'$ is one and the birational map $Y' \dashrightarrow Y^m$ is the outcome of the run of a relative MMP over $X'$, then the exceptional locus of $Y' \dashrightarrow Y^m$ is vertical over $X'$:
in particular, $f'$ and $f^m$ are isomorphic over a dense open set of $X'$, by construction. 
Hence, as $S'$ is a rational section for $f'$, then the strict transform $S^m$ of $S'$ on $Y^m$ is going to be a rational section for $f^m \colon Y^m \to X'$.
\\

We denote by $\mathfrak{E}^m_n$ the set of elliptic fibrations $f^m \colon Y^m \to X'$ that we have just constructed.
We define the set 
\[
\mathfrak{L}_n :=\{ Y_m \; | \; \exists \; f^m \colon Y^m \to X' \in \mathfrak{E}^m_n \}.
\] 
\\

{\bf Step 4}. {\it In this step we show that to prove the theorem it suffices to prove that $\mathfrak{L}_n$ is bounded}.
\newline
Given any elliptic fibration $f \colon Y \to X$ in $\mathfrak{E}_n$, we have constructed an elliptic fibration $f^m \colon Y^m \to X'$ in $\mathfrak{E}^m_n$ such that $Y^m$ is isomorphic to $Y$ in codimension $1$, by Step 2-3.
By construction, $Y^m$ is isomorphic to $Y'$ in codimension $1$ which is in turn isomorphic to $Y$ in codimension $1$.
If $\mathfrak{L}_n$ is bounded, then the set of the $Y$ in $\mathfrak{E}_n$ is bounded up to isomorphisms in codimension $1$, which is what we wish to prove.\\

\noindent
For the remainder of the proof, we fix an elliptic fibration $f^m \colon Y^m \to X'$ in $\mathfrak{E}^m_n$.
We remind the reader that $S_m$ denotes the rational section of $f^m$ constructed in Step 3 as the strict transform of the section $S$ of the corresponding fibration $f \colon Y \to X$.\\

{\bf Step 5}. {\it In this step we show that $(Y^m, S_m)$ is a plt pair}.
\newline 
By adjunction, it suffices to show that $(Y^m, S_m)$ is log canonical and $S_m$ is the only lc center.
Let us consider the normalization of $S_m$, $\nu \colon S_m^\nu \to S_m$. 
On $S_m^\nu$ there is a canonically defined effective divisor $\diff(0)$ such that $K_{S_m^\nu}+\diff(0)=\nu^\ast((K_{Y^m}+S_m)|_{S_m})$, see \cite[\S~4.1]{Kol13}. 
In particular, $K_{S_m^\nu}+\diff(0)$ is relatively big and nef over $X'$.
As $Y^m$ is terminal and $\mathbb{Q}$-factorial, Lemma \ref{section.lemma} implies that
\begin{equation}
\label{exc.diff.eqn}
(f^m \circ \nu)_\ast \diff(0)=0.
\end{equation}
Thus, $(f^m \circ \nu)_\ast (K_{\bar{S}_m^\nu} + \diff(0))=K_{X'}$.
As by Step 1 $X'$ is $\mathbb{Q}$-factorial, the negativity lemma~\cite[Lemma 3.39]{KM} implies that 
\[
K_{S_m^\nu}+\diff(0)=(f_m \circ \nu)^\ast(K_{X'}) -E,
\]
where $E$ is an effective divisor exceptional over $X'$.
Thus, the pair $(S_m^\nu, \diff(0))$ is klt, since $(X', 0)$ is. 
Hence, by inversion of adjunction \cite[Theorem~5.50]{KM} $S_m$ is the only lc center of the pair $(Y^m, S_m)$ which in turn implies that $S_m$ is normal, see \cite[Proposition 5.51]{KM}.\\
    
{\bf Step 6.} {\it In this step we fix a general member $G \in |f_m^\ast(2n+1) H'|$ and we show that $(Y^m, \frac 1 2 S_m+\frac 1 2 G)$ is $\frac 1 2$-klt and $K_{Y^m}+\frac 1 2 S_m+\frac 1 2 G$ is nef and big.
Here $H'$ is the very ample divisor of bounded volume on $X'$ that was constructed at the end of Step 1.}
\newline
Since $|f_m^\ast (2n+1)H'|$ is base point free, Bertini's theorem implies that $(Y^m, S_m+G)$ is log canonical.
On the other hand, $(Y^m, 0)$ is terminal, by Step 3.
As discrepancies of valuations are linear functions of the boundary divisor of a pair, it follows that $(Y^m, \frac 1 2 S_m+\frac 1 2 G)$ is $\frac 1 2$-klt.
As $K_{Y^m} \sim_\mathbb{Q} 0$ and $S_m$ is relatively big, it follows immediately that $K_{Y^m}+\frac 1 2 S_m+\frac 1 2 G$ is big.
As for any $t \in \mathbb{R}$, $K_{Y^m}+t S_m+t G \sim_\mathbb{R} tS_m+tG$, showing nefness of $K_{Y^m}+\frac 1 2 S_m+\frac 1 2 G$ is equivalent to showing nefness of $K_{Y^m}+S_m+G$.
Let us assume that $K_{Y^m}+S_m+G$ is not nef.
Then the cone theorem for log canonical pairs,~\cite[Theorem 1.4]{MR2819675}, implies the existence of a $(K_{Y^m}+S_m+G)$-negative extremal ray $R \subset \overline{\mathrm{NE}}(X)$ which is spanned by the class of a curve $C \subset Y^m$ such that $0 <-((K_{Y^m}+S_m+G) \cdot C) \leq 2n$.
As $G \sim  f_m^\ast (2n+1) H'$ and $H'$ is ample and Cartier, $G \cdot C$ is either $0$ or $\geq 2n+1$ and, consequently, $(K_{Y^m} + S_m) \cdot C < 0$; moreover, as $K_{Y^m} + S_m$ is nef over $X'$, it follows that $C$ is not contained in a fibre of $f^m$, thus, $G \cdot C \geq 2n+1$.
As the class of $C$ spans the extremal ray $R$, applying the cone theorem with the pair $(Y^m, S_m)$, there exists another curve $C'$ whose class spans $R$, $0< -(K_{Y^m} + S_m) \cdot C)  \leq 2n$, and $G \cdot C' \geq 2n+1$.
This leads to a contradiction, as then $(K_{Y^m} + S_m+G) \cdot C'> 0$ while the class of $C'$ was assumed to span a $(K_{Y^m} + S_m+G)$-negative extremal ray.\\
 
{\bf Step 7.} {\it In this step we show that there exists positive integers $C_i= C_i(n, I), i=1, 2$ such that $C_1 \leq (K_{Y^m}+\frac 1 2 S_m+\frac 1 2 G)^{n} \leq C_2$}.
\newline
As $K_{Y^m}+\frac 1 2 S_m+\frac 1 2 G$ is big, the existence of a positive lower bound only depending on $n$ and $I$ for $\vol(Y^m, K_{Y^m}+\frac 1 2 S_m+\frac 1 2 G)$ follows from Step 6 and \cite[Theorem 1.3]{HMX14}.
In Step 6 we have shown that $S_m+G$ is big and nef.
As $G$ is nef, it follows that for any $t \in [0, 1]$,  $G+tS_m$ is nef.
Hence, $(G+tS_m)^n \geq 0$ for any $t \in [0, 1]$ and since $G \sim f^\ast_m (2n+1) H'$ then $G^n =0$ and
\[
(S_m+G)^n=S_m \cdot (S_m+G)^{n-1} + G \cdot (S_m+G)^{n-1}.
\]
Adjunction formula for $S_m$ yields, cf. Step 5,
\[
(K_{Y^m} + S_m)\vert_{S_m} = S_m\vert_{S_m} = K_{S_m} + \diff(0) = (f_m \circ \nu)^\ast (K_{X'}) - E, \ E, \diff(0) \geq 0.
\]
By the canonical bundle formula, cf. Theorem \ref{bundle}, $-K_{X'}$ is effective, thus, $L:=-((f_m \circ \nu)^\ast (K_{X'}) - E)$ is effective, so that
\begin{align*}
& S_m \cdot (S_m+G)^{n-1} = (S_m\vert_{S_m}+G\vert_{S_m})^{n-1} = \vol(S_m, G\vert_{S_m} - L) \leq \vol(S_m, G\vert_{S_m}) \\
& = \vol(S_m, f_m\vert_{S_m}^\ast (2n+1)H') = \vol(X',(2n+1)H') \leq (2n+1)^{n-1} C',
\end{align*}
where $C'= C'(n, I)$ is the positive integer whose existence has been shown in Step 1, and for which $H'^{n-1} \leq C'$.\\
In order to bound $G \cdot (S_m+G)^{n-1}$ we prove the following claim.\\

{\bf Claim}. 
For any $0<k \leq n$,
\[ 
\ G^k\cdot (S_m+G)^{n-k}\leq (n-k)(2n+1)^{n-1} C'.
\]

\begin{proof}
We prove the claim by descending induction on $k$.\\
The claim holds for $k=n$, as $G^n=0$.
Hence we can assume that $k<n$ and that the claim holds for $k+1$.\\
Then, as in the previous paragraph, we have that
\begin{align*}
& G^k\cdot (S_m+G)^{n-k} = G^k\cdot (S_m+G) \cdot (S_m+G)^{n-k-1} = \\
& G^{k+1} \cdot (S_m+G)^{n-k-1} + S_m \cdot G^{k} \cdot (S_m+G)^{n-k-1} \\ 
& \leq (n-k-1)(2n+1)^{n-1} C' +  S_m \cdot G^{k} \cdot (S_m+G)^{n-k-1},
\end{align*}
where the final inequality follows from inductive hypothesis.
Hence, to prove the claim, it suffices to show that 
\[
S_m \cdot G^{k} \cdot (S_m+G)^{n-k-1} \leq (2n+1)^{n-1} C'.
\]
As above, 
\begin{align*}
& S_m \cdot G^{k} \cdot (S_m+G)^{n-k-1}  =
G\vert_{S_m}^{k} \cdot (G\vert_{S_m} - L)^{n-k-1} \leq G\vert_{S_m}^{n-1},
\end{align*}
where the last inequality follows from the fact that $G\vert_{S_m}$ is a semiample divisor and $L$ is effective.
Finally, as we have already seen above, $G\vert_{S_m}^{n-1} \leq (2n+1)^{n-1} C'$.
\end{proof}
\noindent
Thus, to conclude the proof of this step, it suffices to take $C_2 := n(2n+1)^{n-1} H'^{n-1}$.\\

{\bf Step 8.} {\it Conclusion of the proof}.
\newline
For any $Y^m \in \mathfrak{L}_n$, we have constructed a pair $(Y^m, \frac 1 2 S_m+\frac 1 2 G)$ which is $\frac 1 2$-klt. 
The set of coefficients of the boundaries of such pairs is the set $\{\frac 1 2\}$.
By Steps 6-7 and~\cite[Theorem 1.3]{Fil18} $K_{Y^m}+\frac 1 2 S_m+\frac 1 2 G$ is nef and big and its volume lies in a discrete set $J$ of $[C_1, C_2] \subset \mathbb{R}_{>0}$, where $C_i=C_i(n, I)$.
The set $J$ is discrete and limited in a closed interval, hence, finite.
Thus, the set of pairs $(Y^m, \frac 1 2 S_m+\frac 1 2 G)$ constructed in Step 6, where $Y^m \in \mathfrak{L}_n$, is bounded by \cite[Theorem 6]{MST16}.
In particular, the set $\mathfrak{L}_n$ itself is bounded and, by Step 4, this concludes the proof.
\end{proof}

\begin{lemma}\label{section.lemma}
Let $g \colon T \to B$ be a surjective morphism of normal quasi-projective varieties with $\dim T -\dim B=1$.
Assume that $(T, 0)$ is terminal and that there exists a rational section $s \colon B \dashrightarrow T$.
Let $S$ be the Zariski closure of $s(B)$ and let $\nu\colon S^\nu \to S$ be the normalization of $S$. 
Then, $\diff_{S^\nu}(0)$ is exceptional over $B$.
\end{lemma}

\begin{proof}
Since $s$ is a rational section, the finite part in the Stein factorization of $g$ is an isomorphism.
Hence we can assume that $g$ is a contraction.\\
Since $T$ is terminal, then it is smooth in codimension two.
In particular, $\omega_T(S)$ is locally free at each codimension one point of $S$.
Let $P \in S$ be a codimension one point such that $g\vert_S(P)=Q \in B$ and $Q$ is codimension one. 
As $T$ is smooth at $P$, it suffices to show that $S$ is normal at $P$, since then $\nu$ would be an isomorphism locally around $P$ and the coefficient of $\diff_{S^\nu}(0)$ at the codimension one point $\nu^{-1}(P)$ would be $0$.
\\
Considering the birational morphism $g\vert_S \circ \nu \colon S^\nu \to B$, the generic point $P'$ of the strict transform of the closure of $Q$ on $S^\nu$ is the unique codimension one point which is mapped to $Q$. 
Hence, $g\vert_S^{-1}(Q)=P$ and, furthermore, $\nu^{-1}(P)=P'$.
Since $B$ is normal quasi-projective, the map $s$ is well defined at $Q$, and $s(Q)=P$ since $g\vert_S\circ s$ is the identity around $Q$ and $P$ is the only point above $Q$.
Then there exists a lift $s_\nu \colon B \dashrightarrow S^\nu$ which is well defined at $Q$ and such that $\nu \circ s_\nu=s$ and $s_\nu(Q)=P'$ by construction.
But then, since $S^\nu$ is smooth at $P'$ and $g\vert_S \circ \nu \circ s_\nu$ is the identity of $B$ and it is a morphism around $P$, it follows that $s_\nu$ is an isomorphism locally around $P$ and the same holds for $s$, which implies that normality of $S$ at $P$.
\end{proof}


\section{Proof of the Corollaries}
\label{sec.cor.proof}

In this section we prove the corollaries stated in the Introduction.

\begin{proof}[Proof of \ref{bounded.lcy.mfs.thm}]
The canonical bundle formula and Theorem~\ref{birk.thm} imply that there exist a positive real number $\delta=\delta(n, I) >0$ an effective divisor $\Gamma_Z$ on $Z$ such that $(Z, \Gamma_Z)$ is a $\delta$-klt Calabi--Yau pair.
As $-K_Z$ is big, Theorem \ref{bab.thm} implies that $Z$ belongs to a bounded family.
Finally, Theorem \ref{bounded.bases} concludes the proof.
\end{proof}

\begin{proof}[Proof of \ref{picard}]
Let $\mathfrak{D}$ be the set of pairs $(X, \Delta)$ satisfying the hypotheses of the statement of the corollary.
By Theorem~\ref{main}, there exists a bounded set of pairs $\mathfrak{D}'$ such that for any pair $(X, \Delta) \in \mathfrak{D}$
there exists a pair $(X', \Delta') \in \mathfrak{D}'$ such that $(X, \Delta)$, $(X', \Delta')$ are isomorphic in codimension one, so that $\rho(X)=\rho(X')$, and for any $m \in \mathbb{N}$, $h^0(X, m(K_X+\Delta))= h^0(X, m(K_{X'}+\Delta'))$.
Moreover, by~\cite[Theorem 6]{MST16}, we can assume that the pairs in $\mathfrak{D}'$ are $\mathbb{Q}$-factorial, so that $\rho(X') \leq h^2(X', \mathbb{Q})$.
Hence, it suffices to show that the conclusion in the statement of the corollary holds for the pairs in $\mathfrak{D}''$.
Furthermore, since $(X, \Delta) \in \mathfrak{D}$ are not of product type, Definition~\ref{defproduct} readily implies that the same holds for $(X', \Delta')\in \mathfrak{D}'$, as these pairs are isomorphic in codimension one.
\\
By definition of boundedness, there exists a pair $(Z, D) \to T$ over a base $T$ of finite type such that for any pair $(X', \Delta') \in \mathfrak{D}''$, there exists $t \in T$ such that the pair $(Z_t, D_t)$ is isomorphic to $(X', \Delta')$.
By substituting $T$ with the Zariski closure of the set of points $t \in T$ for which $(Z_t, D_t)$ is isomorphic to a pair $(X', \Delta') \in \mathfrak{D}'$, we can then assume that the set of such $t$ is Zariski dense in $T$. 
\\
Up to decomposing $T$ into a finite union of locally closed subset, we may assume that there exists a log resolution $\psi\colon (Z', D') \to Z$ of $(Z, D)$, where $D'$ is the sum of the strict transform of $D$ and the exceptional divisor of $\psi$.
Furthermore, up to decomposing $T$ further into a finite union of locally closed subset, we may also assume that for any $t \in T$ $(Z'_t, D'\vert_{Z'_t})$ is a log resolution of $(Z_t, D\vert_{Z_t})$.
In particular, for any $t \in T$, for all $m > 0$
\begin{align*}
H^0(Z_t, \mathcal{O}_{Z_t}(m(K_{Z_t}+D\vert_{Z_t}))) = 
H^0(Z'_t, \mathcal{O}_{Z'_t}(m(K_{Z'_t}+D'\vert_{Z'_t}))).
\end{align*}
Then~\cite[Theorem 4.2]{MR3779687} implies that for any  connected component $\overline{T}$ of $T$ $h^0(Z_t, \mathcal{O}_{Z_t}(m(K_{Z_t}+D\vert_{Z_t})))$ is independent of $t \in \overline{T}$, for all $m >0$.
At this point, we discard those connected components of $T$ that do not points $t$ such that $(Z_t, D_t)$ is isomorphic to one of the pairs in $\mathfrak{D}'$.
By construction then, for a connected component $\overline{T}_i$ of $T$, there exists $\bar{t}_i \in \overline{T}_i$ and a positive integer $\bar{m}_i$ such that $\bar{m}_i(K_{Z_{\bar{t}_i}} + D\vert_{Z_{\bar{t}_i}}) \sim 0$;
thus, by construction, 
\begin{align}
\label{inv.plurig.cy.eqn}
h^0(Z_{\bar{t}_i}, \mathcal{O}_{Z_{\bar{t}_i}}(\bar{m}_i(K_{Z_{\bar{t}_i}} + D\vert_{Z_{\bar{t}_i}}))=1, \forall \bar{t}_i \in \overline{T}_i.
\end{align} 
Hence, it suffices to define $m_0$ to be the maximum of the positive integers $\bar{m}_i$ just defined; 
$m_0$ is well defined as $T$ has only finitely many connected components, being of finite type, by definition of boundedness.
\\
We turn now to showing that the Picard ranks of the fibers of $(Z, D) \to T$ are bounded.
Since for any $(X', \Delta') \in \mathfrak{D}'$, $\rho(X') \leq h^2(X', \mathbb{Q})$,
It suffices to show that there exists a positive integer $\rho$ such that $h^2(Z_t, \mathbb{Q}) \leq \rho$ for any $t \in T$.
By Verdier's generalization of Ehresmann's theorem \cite[Corollaire~5.1]{Ver}, up to decomposing $T$ into a disjoint union of finitely many locally closed subvarieties, we can assume that $Z \to T$ is a locally trivial topological fibration. 
Hence, on each connected component $\overline{T}_i$ of $T$, there exists $\bar{\rho}_i$ such that 
\[
h^2(Z_t, \mathbb{Q})=\bar{\rho}_i, \ \forall t \in \overline{T}_i.
\] 
Hence, it suffices to define $\rho$ to be the maximum of the positive integers $\bar{\rho}_i$ just defined.
\end{proof}

\begin{proof}[Proof of \ref{cy.fibration}]
Let $\pi \colon Y \rightarrow X$ be an elliptic fibration satisfying the hypotheses of the statement of the corollary.
Using the canonical bundle formula, cf. Theorem~\ref{bundle}, 
by Theorem~\ref{semiample} and Lemma~\ref{acc.lemma}, there exists a divisor $\Delta$ on $X$ such that $(X, \Delta)$ is a klt Calabi--Yau pair and the coefficients of $\Delta$ are in a DCC set $J=J(n, I)$; 
by Theorem~\ref{hmx_1.5.thm}, we can assume that the set $J$ is finite.
By Lemma~\ref{lem:cy.not.prod.type}, as $Y$ is not of product type, then $\Delta \neq 0$ and $(X, \Delta)$ is not of product type.
The result then follows from Theorem~\ref{main}.
\end{proof}

\begin{proof}[Proof of \ref{effective.nonvanishing}]
Then the result follows from Theorem~\ref{bundle}, Corollary~\ref{cy.fibration} and Corollary~\ref{picard}.
\end{proof}

\begin{proof}[Proof of \ref{max.var.cor}]
By Theorem~\ref{hmx_1.5.thm}, we can assume that the coefficients of $\Delta$ belong to a finite set $I_0 \subset I$.
\\
As $f$ has maximal variation, then the moduli part $M_Z$ induced by the canonical bundle formula, Theorem~\ref{bundle}, for $(X, \Delta)$ is big, see Remark~\ref{rmk.max.var}.
In particular, $-K_Z$ is big, because $K_Z+B_Z+M_Z \sim_{\mathbb{R}}0$ and $B_Z$ is effective.
\\
As $\Delta$ is big over $Z$, for $0< \eta \ll 1$ we can run the relative $(K_X+(1+\eta)\Delta)$-MMP and this terminates with an ample model $f' \colon X' \to Z$ on which the strict transform $\Delta'$ of $\Delta$ is ample over $Z$.
We first show that the set of pairs $(X', \Delta')$ forms a bounded family.
\\
As the coefficients of $\Delta'$ belong to $I_0$, then given a general fibre $F'$ of $f'$, the pair $(F', \Delta'\vert_{F'})$ belongs to a bounded family by Corollary~\ref{elc} and Theorem~\ref{bab.thm}.
Theorem~\ref{birk.thm} together with Remark~\ref{rmk.bab.sing.base} imply that there exists $\delta=\delta(n, I_0)$ and a big effective divisor $\Gamma_Z$ on $Z$ such that $(Z, \Gamma_Z)$ is a $\delta$-klt Calabi--Yau pair and $\Gamma_Z$ is big, since $-K_Z$ is big.
Thus, Theorem~\ref{bab.thm} implies that the set $\mathfrak{F}$ of varieties $Z$ forms a bounded family.
As $-K_{X'}$ is $f'$-ample and $\mathfrak{F}$ forms a bounded family, Theorem~\ref{bounded.bases} implies that the set $\mathfrak{D}''$ of pairs $(X', \Delta')$ forms a bounded family.
\\
As the set of pairs $(X', \Delta')$ forms the bounded family $\mathfrak{D}''$,~\cite[Proposition~2.5]{HX14} implies that there exists a bounded family $(Z,D)\rightarrow T$ such that for any pair $(X',\Delta')$ and for any set $\{E_1, \dots, E_k\}$ of exceptional divisors of discrepancy at most $0$ over $(X', \Delta')$, there exists $t\in T$ and a birational morphism $\mu_t : Z_t \rightarrow X'$ which extracts precisely $\{E_1, \dots, E_k\}$ and $K_{Z_t}+D_t=\mu_t^* (K_{X'} +\Delta')$. 
Since $K_X +\Delta \sim_\rr 0$ and $K_{X'} +\Delta' \sim_\rr 0$, all the exceptional divisors for $r \colon X \dashrightarrow X'$ have discrepancy $\leq 0$ with respect to $(X', \Delta')$.
In particular, $(X, \Delta)$ is birational in codimension one to a pair $(Z_s, D_s)$ for some closed point $s \in T$.
Defining $\mathfrak{D}'$ to be the set of pairs $(Z_s, D_s)$ just constructed terminates the proof.
\end{proof}

\begin{proof}[Proof \ref{hodge}]
By Theorem~\ref{elliptic} there exists a family $\mathcal{X}\rightarrow \mathcal{T}$ such that any elliptic Calabi-Yau manifold with a section is birational in codimension $1$ to a fiber $\mathcal{X}_s$ of the family. 
Let $\mathcal{Y}\rightarrow \mathcal{X}$ be a log resolution. 
Since the generic fiber of $\mathcal{Y}$ is smooth, there exists a Zariski open subset $U$ of $\mathcal{T}$ such that the $E(\mathcal{Y}_s)$ does not depend on $s\in U$, where $E$ is the Euler function of a smooth variety. 
By the definition of the stringy $E$-function~\cite{Bat99}, we have that $E(\mathcal{X}_s)$ is constant over $U$ as well. 
By Noetherian induction, it follows that we obtain only a finite number of stringy $E$-functions from the fibers of $\mathcal{X}\rightarrow \mathcal{T}$. 
\newline
Let $X$ be a smooth elliptic Calabi-Yau with a section. By Theorem \ref{elliptic}, $X$ is K-equivalent to $\mathcal{X}_s$ for some $s$, and~\cite[Theorem~2.7]{Vey01} implies that they have the same $E$-function. 
In particular, there are finitely many possibilities for the $E$-function of an elliptic Calabi-Yau with a section. 
The Hodge numbers are precisely the coefficients of those functions and thus, they are uniformly bounded. 
See~\cite{Vey01} for more details. 
\end{proof}




\begin{thebibliography}{ELMNPM}

\bibitem[Ale94]{Ale94} V. Alexeev, \textit{Boundedness and $K^{2}$ for log surfaces}, Internat. J. Math. 5 (1994), no. 6, 779-810. 

\bibitem[Amb04]{Amb04} F. Ambro, \textit{Shokurov's boundary property}, 
J. Differential Geom. 67 (2004), no. 2, 229-255. 

\bibitem[Amb05]{Amb05} F. Ambro, \textit{The moduli b-divisor of an lc-trivial fibration},
Compos. Math. 141 (2005), no. 2, 385-403. 

\bibitem[Bat99]{Bat99} V.  Batyrev, \textit{Stringy Hodge numbers of varieties with Gorenstein canonical singularities}, Proc. Taniguchi Symposium 1997, In ``Integrable Systems and Algebraic Geometry, Kobe/Kyoto 1997'', World Sci. Publ. (1999), 1-32.

\bibitem[Bea83]{MR730926} A. Beauville, \textit{Vari\'et\'es {K}\"ahleriennes dont la premi\`ere classe de {C}hern est nulle}, J. Differential Geom. 18 (1983),
no. 4, 755-782.

\bibitem[Bir16a]{Bir16} C. Birkar, \textit{Singularities on the base of a Fano type fibration}, 
J. Reine Angew. Math. 715 (2016), 125-142. 

\bibitem[Bir16b]{bir16bab} 
C.~Birkar, \textit{Singularities of linear systems and 
boundedness of Fano varieties},
Ann. of Math. (2) 193 (2021), no. 2, 347-405.

\bibitem[Bir18]{birkar.lcy.fibr}
C.~Birkar, \textit{Log Calabi--Yau fibrations}, 2018,
arXiv e-print, arXiv:1811.10709v2. 

\bibitem[Bir19]{bab16a} 
C.~Birkar, \textit{Anti-pluricanonical systems on Fano varieties}, Ann. of Math. (2) 190 (2019), no. 2, 345--463.

\bibitem[BCHM10]{BCHM} C. Birkar, P. Cascini, C. Hacon and J. M\textsuperscript{c}Kernan,
\textit{Existence of minimal models for varieties of log general type},
J. Amer. Math. Soc. \textbf{23} (2010), 405-468.

\bibitem[BDCS20]{bds} 
C.~Birkar, G.~Di~Cerbo and R.~Svaldi,
\textit{Boundedness of elliptic Calabi--Yau varieties with a rational section}, 2020,
arXiv e-print, 
arXiv:2010.09769v1.


\bibitem[BZ16]{bir.zhang}
C. Birkar and D.-Q. Zhang, 
\textit{Effectivity of Iitaka fibrations and pluricanonical systems of polarized pairs},
Publ. Math. IHES (2016) 123: 283. 


\bibitem[C$^+$21]{rccy3}
W.~Chen, G.~Di~Cerbo, J.~Han, C.~Jiang and R.~Svaldi,
\textit{Birational boundedness of rationally connected Calabi--Yau 3-folds},
Adv. Math. 378 (2021), 107541.

\bibitem[Fil20a]{Fil18a} S. Filipazzi, \textit{On a generalized canonical bundle formula and generalized adjunction}, 
Ann. Sc. Norm. Super. Pisa Cl. Sci. (5)
Vol. XXI (2020), 1187--1221.

\bibitem[Fil20b]{Fil18} S. Filipazzi, \textit{Some remarks on the volume of log varieties},
Proc. Edinb. Math. Soc. (2) 63 (2020), no.2, 314--322.

\bibitem[Fil20c]{Fil20} S. Filipazzi, \textit{On the boundedness of $n$-folds with $\kappa(X)=n-1$}, 2020,
arXiv e-print, 
arXiv:2005.05508v2.

\bibitem[FS20]{FS20}
S. Filipazzi and R. Svaldi, \textit{Invariance of plurigenera and boundedness for generalized pairs}, Mat. Contemp. 47 (2020), 114--150.

\bibitem[Flo14]{Flo}
E.~Floris,
\textit{Inductive approach to effective b-semiampleness},
Int. Math. Res. Not. IMRN 2014, no. 6, 1465--1492.

\bibitem[Fuj11a]{Fuj11} O.~Fujino, \textit{On Kawamata's theorem}, Classification of algebraic varieties, 305--315, EMS Ser. Congr. Rep., Eur. Math. Soc., Z\"{u}rich, 2011.

\bibitem[Fuj11b]{MR2819675}
O.~Fujino, \textit{Non-vanishing theorem for log canonical pairs}, 
J. Algebraic Geom., 20 (2011), no. 4, 771--783.

\bibitem[FG12]{MR2944479} O. Fujino and Y. Gongyo
\textit{On canonical bundle formulas and subadjunctions},
Michigan Math. J. 61 (2012), no. 2, 255--264. 

\bibitem[GL13]{GL} Y. Gongyo and B. Lehmann, \textit{Reduction maps and minimal model theory}, 
Compos. Math. 149 (2013), no. 2, 295-308. 

\bibitem[Gro94]{Gross} M. Gross, \textit{A finiteness theorem for elliptic Calabi-Yau threefolds}, 
Duke Math. J. 74 (1994), no. 2, 271-299. 

\bibitem[HMX13]{HMX13} C. Hacon, J. M\textsuperscript{c}Kernan and C. Xu, \textit{On the birational automorphisms of varieties of general type},  Ann. of Math. (2) 177 (2013), no. 3, 1077-1111. 

\bibitem[HMX14]{HMX14} 
C. Hacon, J. M\textsuperscript{c}Kernan and C. Xu, 
\textit{ACC for log canonical thresholds},
Ann. of Math. (2) 180 (2014), no. 2, 523-571.

\bibitem[HMX18]{MR3779687}
C. Hacon, J. M\textsuperscript{c}Kernan and C. Xu, 
\textit{Boundedness of moduli of varieties of general type},
J. Eur. Math. Soc. (JEMS), 20 (2018), no. 4, 865--901.

\bibitem[HX13]{HX} C. Hacon and C. Xu, 
\textit{Existence of log canonical closures}, 
Invent. Math. 192 (2013), no. 1, 161--195. 

\bibitem[HX15]{HX14} C. Hacon and C. Xu, 
\textit{ Boundedness of log Calabi-Yau pairs of Fano type}, Math. Res. Lett. 22 (2015), no.6, 1699-1716.

\bibitem[HL19]{HL} J.~Han and W.~Liu, \textit{On a generalized canonical bundle formula for generically finite morphisms}, 2019,
arXiv e-print, arXiv:1905.12542v3,
to appear in Annales de l'Institut Fourier.

\bibitem[Har77]{Har} R. Hartshorne, Algebraic Geometry. Graduate Texts In Mathematics, No. 52. \textit{Springer-Verlag, New York-Heidelberg},
1977.

\bibitem[Jia18]{Jia15} C. Jiang, \textit{On birational boundedness of Fano fibrations},  Amer. J. Math. 140 (2018), no. 5, 1253--1276.

\bibitem[Kol93]{K93} J. Koll\'{a}r, 
\textit{Effective base point freeness}.
Math. Ann. 296 (1993), no. 4, 595--605.

\bibitem[Kol96]{Kol96} J. Koll\'ar, \textit{Rational curves on algebraic varieties},
  Ergebnisse der Mathematik und ihrer Grenzgebiete. 3. Folge. A Series of Modern Surveys in Mathematics, {\bf 32}, {\it Springer-Verlag, Berlin}, 1996.

\bibitem[Kol97]{Kol97} J. Koll\'ar, \textit{Singularities of pairs}, Algebraic geometry---{S}anta {C}ruz 1995,
Proc. Sympos. Pure Math. (1997), 62, 221--287.

\bibitem[Kol13]{Kol13} J. Koll\'ar, \textit{Singularities of the minimal model program},
with a collaboration of S\'andor Kov\'acs. Cambridge Tracts in Mathematics, 200. Cambridge University Press, Cambridge, 2013.

\bibitem[KL09]{KL} J. Koll\'ar and M. Larsen, \textit{Quotients of Calabi-Yau varieties}, 
Algebra, arithmetic, and geometry: in honor of Yu. I. Manin. Vol. II, 179--211,
Progr. Math., 270, Birkh\"auser Boston, Inc., Boston, MA, 2009. 

\bibitem[KM98]{KM} J. Koll\'ar and S. Mori, \textit{Birational geometry of algebraic varieties},
Cambridge Tracts in Mathematics, vol. 134, Cambridge University Press, 1998. 

\bibitem[K$^+$92]{MR1225842} 
J. Koll\'ar et al., 
{\it Flips and abundance for algebraic threefolds}. 
Papers from the Second Summer Seminar on Algebraic Geometry held at the University of Utah, Salt Lake City, Utah, August 1991. 
Ast\'erisque No. 211 (1992). 
Soci\'et\'e Math\'ematique de France, Paris, 1992. 
pp. 1--258.

\bibitem[Lai11]{Lai} C.-J.~Lai,  {\it Varieties fibered by good minimal models}, 
Math. Ann. 350 (2011), no. 3, 533--547.

\bibitem[Laz04]{Laz} R. Lazarsfeld, Positivity in Algebraic Geometry I, II,
Ergebnisse der Mathematik und ihrer Grenzgebiete. 3. Folge. A series
of Modern Survays in Mathematics \textbf{48},
\textit{Springer-Verlag, Berlin}, 2004.

\bibitem[MST20]{MST16}
D. Martinelli, S. Schreieder and L. Tasin,
\textit{On the number and boundedness of log minimal models of general type}, Ann. Sci. \'Ec. Norm. Sup\'er. (4) 53 (2020), no. 5, 1183--1210.

\bibitem[MP04]{MP}
J. M\textsuperscript{c}Kernan and Y. Prokhorov,
\textit{Threefold thresholds}, 
Manuscripta Math. 114 (2004), no. 3, 281--304. 

\bibitem[OT15]{MR3329200}
K. Oguiso and T. T. Truong,
\textit{Explicit examples of rational and {C}alabi-{Y}au threefolds with primitive automorphisms of positive entropy},
J. Math. Sci. Univ. Tokyo 22 (2015), no. 1, 361--385.

\bibitem[PS09]{PS} Y. Prokhorov and V. Shokurov, \textit{Towards the second main theorem on complements},
J. Algebraic Geom. 18 (2009), no. 1, 151-199.

\bibitem[TW15]{TW} W. Taylor and Y.-N. Wang, \textit{The F-theory geometry with most flux vacua}, J. High Energ. Phys. 2015, 1--21 (2015).

\bibitem[Ver76]{Ver}
J.-L. Verdier, \textit{Stratifications de {W}hitney et th\'{e}or\`eme de {B}ertini-{S}ard}, Invent. Math. 36 (1976), 295--312.

\bibitem[Vey01]{Vey01} W. Veys, \textit{$\zeta$ functions and ``Kontsevich invariants'' on singular varieties}, 
Canadian journal of mathematics-journal canadien de mathematiques vol. 53 issue. 4 (2001), 834-865.

\end{thebibliography}
\end{document}